\documentclass[11pt,twoside]{amsart}
\usepackage{latexsym,amssymb,amsmath}

\numberwithin{equation}{section}
\newtheorem{theorem}{Theorem}[section]
\newtheorem{lemma}[theorem]{Lemma}
\newtheorem{proposition}[theorem]{Proposition}
\newtheorem{corollary}[theorem]{Corollary}
\newtheorem{conjecture}[theorem]{Conjecture}

\theoremstyle{definition}
\newtheorem{definition}[theorem]{Definition}

\newtheorem{remark}[theorem]{Remark}
\newtheorem{example}[theorem]{Example}

\newcommand{\NN}{\mathbb{N}}
\newcommand{\ZZ}{\mathbb{Z}}
\newcommand{\QQ}{\mathbb{Q}}
\newcommand{\PP}{\mathbb{P}}
\newcommand{\TT}{\mathbb{T}}
\newcommand{\xx}{\mathbf{x}}

\newcommand{\ww}{\mathbf{w}}
\newcommand{\C}{\mathcal{C}}
\newcommand{\G}{\mathcal{G}}
\renewcommand{\H}{\mathcal{H}}

\newcommand{\supp}{\operatorname{supp}}
\newcommand{\ev}{\rm ev}
\newcommand{\reg}{\operatorname{reg}}
\newcommand{\Ker}{\operatorname{Ker}}

\theoremstyle{theorem}

\newtheorem*{claim1}{Claim 1}
\newtheorem*{claim2}{Claim 2}
\newtheorem*{claim3}{Claim 3}
\newtheorem*{claim4}{Claim 4}

\renewcommand{\labelenumi}{(\arabic{enumi})}
\newcommand{\set}[1]{\left\{ #1 \right \} }
\newcommand{\smset}[1]{\{ #1  \} }

\newcommand{\ls}[1]{\left | #1 \right | }
\newcommand{\smls}[1]{| #1 | }

\newcommand{\rt}{\rightarrow}
\newcommand{\lrt}{\longrightarrow}

\begin{document}

%%%%%%%%%%%%%%%%%%%%%%%%%%%%%%%%%%%%%%%%%%%%%%%%%%%%%%%%%%%%%%%%%%%%%

\title[Vanishing ideals over graphs and even cycles]{Vanishing
ideals over graphs and even cycles}

\thanks{The first author was partially supported by CMUC and FCT (Portugal), through
European program COMPETE/FEDER. The second author is a member of the Center for Mathematical Analysis,
Geometry and Dynamical Systems. The third author was partially supported by SNI}

\author{Jorge Neves}
\address{CMUC, Department of Mathematics, University of Coimbra
3001-454 Coimbra, Portugal.
}
\email{neves@mat.uc.pt}

\author{Maria Vaz Pinto}
\address{
Departamento de Matem\'atica\\
Instituto Superior T\'ecnico\\
Universidade T\'ecnica de Lisboa\\
Avenida Rovisco Pais, 1\\
1049-001 Lisboa, Portugal
}
\email{vazpinto@math.ist.utl.pt}

\author{Rafael H. Villarreal}
\address{
Departamento de
Matem\'aticas\\
Centro de Investigaci\'on y de Estudios
Avanzados del
IPN\\
Apartado Postal
14--740 \\
07000 Mexico City, D.F.
}
\email{vila@math.cinvestav.mx}
%\urladdr{http://www.math.cinvestav.mx/$\sim$vila/}

%\keywords{Evaluation codes, parameterized codes, algebraic toric set,
% vanishing ideals, parameters of a code, finite fields,
%minimum distance, degree,
%regularity, Hilbert function.}

\begin{abstract} Let $X$ be an algebraic toric set in a projective
space over a 
finite field. We study the vanishing ideal, $I(X)$, of $X$ and show
some useful degree bounds  
for a minimal set of generators of $I(X)$. 
We give an explicit combinatorial description of a set of generators
of $I(X)$, when $X$ is the algebraic toric set associated to an even
cycle or to a 
connected bipartite graph with pairwise vertex disjoint even cycles.
In this 
case, a formula for the regularity of $I(X)$ is given. We show an 
upper bound for this invariant, when $X$ is associated to 
a (not necessarily connected) bipartite graph. The upper bound is
sharp if the graph is connected. We are able to show a formula for 
the length of the parameterized linear code associated with any
graph, 
in terms of the number of bipartite and non-bipartite components. 
\end{abstract}

\maketitle

\section{Introduction}
Let $\mathbb{P}^{s-1}$ be a projective space over a finite field
$\mathbb{F}_q$. 
An {\it evaluation code}, also known as a {\it generalized
Reed-Muller code}, is a linear code
obtained by evaluating the linear space of homogeneous $d$-forms on a
set of points 
$X\subset \PP^{s-1}$ (see Definition~\ref{definition: evaluation
code}). A linear code obtained in this way,  denoted by
$C_X(d)$, has length $\ls{X}$.  Evaluation codes have been the object
of much attention in recent 
years. To describe their basic 
parameters (length, dimension and minimum distance), many authors 
have been using tools coming from  Algebraic Geometry and Commutative
Algebra, see
\cite{delsarte-goethals-macwilliams,
duursma-renteria-tapia,
gold-little-schenck,GRT,ci-codes,sorensen,tohaneanu}. Let $\TT^{s-1}$ be a
 projective torus in $\mathbb{P}^{s-1}$. A {\it parameterized linear
 code} is a special type 
of generalized Reed-Muller code obtained when $X\subset
\TT^{s-1}\subset \PP^{s-1}$ is parameterized by a set of monomials  
(see Definition~\ref{definition: parameterized linear code}), in this
case $X$ is called an {\it algebraic toric set} because it generalizes
the notion of a projective torus. Parameterized linear codes 
were introduced and studied in \cite{algcodes}. The extra structure
on $X$ yields  
alternative methods to compute the basic parameters of $C_X(d)$. 

In this article we focus on linear codes parameterized by the edges
of a graph $\G$  
(see Definition~\ref{definition: parameterized code associated to a
graph}). For the study 
of algebraic toric sets parameterized by the edges of a {\it
clutter}, which is a natural generalization 
of the concept of graph, we refer the reader to
\cite{ci-codes,vanishing}.  
Not much is known about the parameterized linear codes associated to
a general graph. 
The first results in this direction appear in \cite{GR}, where the
length, dimension and minimum distance 
of the codes associated to complete bipartite graphs are computed. In 
\cite{algcodes}, one can find a formula for the length of the code 
associated to a connected graph (see this formula in
Proposition~\ref{proposition: order of X for G connected}) and also a
bound 
for the minimum distance of the code associated to a connected
non-bipartite graph. 

An important algebraic invariant associated to a parameterized linear
code is the regu\-la\-rity 
of the ring $S/I(X)$, where $S$ is the coordinate ring of
$\PP^{s-1}$, {\it i.e.}, a polynomial 
ring in $s$ variables, and $I(X)$ is the vanishing ideal of $X$ 
(see Definition~\ref{definition: regularity}). The knowledge of the
regularity of $S/I(X)$ is important for applications to 
coding theory: for $d\geq \reg S/I(X)$ the code $C_X(d)$ coincides
with the underlying vector 
space $\mathbb{F}_q^{|X|}$ and has, 
accordingly, minimum distance equal to $1$. In \cite[Corollary
2.31]{deg-and-reg} the authors  
give bounds for the regularity of $S/I(X)$, when $X$ is the algebraic
toric set associated to  
a connected bipartite graph. In \cite{codes over cycles} a bound is
given for the minimum distance 
of the codes associated to a graph isomorphic to a cycle of even length, 
as well as another bound for $\reg S/I(X)$ in this case.

%The structure of the paper is as follows. 
The contents of this paper are as follows. In Section~\ref{sec:
prelimiaries}, we recall the necessary background. To the best of our
knowledge, 
there is no information available on
the parameterized codes arising from disconnected graphs. If $\G$ is
an arbitrary graph, in Section~\ref{length}, Theorem~\ref{theorem:
computation of the length 
for any graph}, we show our first main result, an explicit formula
for the length of 
$C_X(d)$ in terms of the number of bipartite and non-bipartite
connected components of the graph.  

An earlier result of \cite{algcodes} shows that the vanishing ideal
$I(X)$ is minimally generated by a finite set of homogeneous binomials. 
In Section~\ref{sec: generators of I(X)}, 
we study $I(X)$ for an
arbitrary algebraic toric set $X$ and show some useful degree bounds 
for a minimal set of generators of $I(X)$ (see Theorem~\ref{lemma: first
reduction} and Proposition~\ref{proposition: bound on the degrees of
generators}). If the graph $\G$ is an even cycle, another main result
of this article 
is an explicit
combinatorial description of a generating set for $I(X)$ consisting
of 
binomials (see Theorem~\ref{theorem: the generators of I(X)}). This
result is generalized to any connected bipartite graph whose cycles
are vertex disjoint 
(see Theorem~\ref{theorem: I(X) for almost general graph}). 
We give examples of bipartite graphs not satisfying this assumption for which 
$I(X)$ is not generated by the set prescribed in 
Theorem~\ref{theorem: I(X) for almost general graph} (see
Example~\ref{example: two offending graphs}).  

If the graph $\G$ is an even cycle of length $2k$, using our
description of a generating set for 
$I(X)$, we derive the following formula for the regularity:
\[
 \reg S/I(X) = (q-2)(k-1)
\]
(see Theorem~\ref{theorem: regularity of even cycles}). Then, we give
the following upper bound for the regularity of $S/I(X)$ for a
general (not necessarily connected) bipartite 
graph with $s$ edges and $m$ cycles, with disjoint edge sets, of
orders $2k_1,\dots,2k_m$:
\begin{equation*}
\reg S/I(X)  \leq  (q-2)\big(s-{\textstyle \sum}_{i=1}^m k_i -1\big)
\end{equation*}
(see Theorem~\ref{theorem: bound on regularity for a general graph}).
In Corollary~\ref{corollary: regularity for almost general graphs},
we show that this estimate  
is the actual value of $\reg S/I(X)$ if $\G$ is a connected bipartite
graph with $s$ edges and with exactly $m$ even cycles, with disjoint
vertex sets, of orders $2k_1,\dots,2k_m$.

The computational algebra techniques of \cite{algcodes} played an
important role in discovering some of the results, conjectures, and
examples of this paper. Using the computer algebra system
\emph{Macaulay}$2$ \cite{mac2} and the 
results of \cite{algcodes}, one can compute the reduced Gr\"obner
basis, the degree and the regularity of a vanishing ideal $I(X)$ of 
an algebraic toric set $X$ over a finite field $\mathbb{F}_q$. This
allows us to study and to 
gain insight on the algebraic invariants of a
vanishing ideal that are useful in algebraic coding theory. 

For all unexplained
terminology and additional information,  we refer to 
\cite{Boll} (for graph theory), \cite{EisStu} (for the theory of
binomial ideals), 
\cite{eisenbud-syzygies,harris,Sta1} (for commutative algebra and the
theory of Hilbert functions), and 
\cite{stichtenoth,tsfasman} (for the theory of
linear codes and evaluation codes).

%%%%%%%%%%%%
%%%%%%%%% SECTION
%%%%%%%%%%%%

\section{Preliminaries}\label{sec: prelimiaries}

Let $K=\mathbb{F}_q$ be a finite field of order $q$ and fix $s$ a
positive integer. Recall that the {\it projective space\/} of
dimension $s-1$ over $K$, denoted by $\mathbb{P}^{s-1}$, is the quotient space
$(K^{s}\setminus\{0\})/\sim $ where two vectors $\xx_1$, $\xx_2$ in
$K^{s}\setminus\{0\}$ 
are equivalent if $\xx_1=\lambda{\xx_2}$ for some $\lambda\in
K^*=K\setminus \set{0}$. Denote by  
$\TT^{s-1}$ the subset of $\PP^{s-1}$ given by 
$$\TT^{s-1}=\set{[\xx]=[(x_1,\dots,x_s)]\in \PP^{s-1}: x_1\cdots x_s
\not =0},
$$ 
where $[\xx]$ is the equivalent class of $\xx$. The 
\emph{projective torus} $\TT^{s-1}$ is an Abelian group under
componentwise multiplication and 
is isomorphic to the standard $(s-1)$-dimensional torus,
$(K^*)^{s-1}$, over $K$. 
\smallskip

Consider \mbox{$S=K[t_1,\ldots,t_s]=\bigoplus_{d=0}^\infty S_d$},
a polynomial ring over the field $K$ with the standard grading.
Given a nonempty set of points $X=\set{[\xx_1],\dots,[\xx_m]}\subset
\TT^{s-1}\subset \PP^{s-1}$ and 
letting $f_0=t_1$, consider, for each $d$, the map: 
${\rm ev}_d\colon S_d\rightarrow K^{\ls{X}}$ given by 
\begin{equation}\label{eq: ev-map}
f\mapsto
\left(\frac{f(\xx_1)}{f_0^d(\xx_1)},\ldots,
\frac{f(\xx_m)}{f_0^d(\xx_m)}\right),\quad
\forall\:{f \in S_d}.  
\end{equation}
For each $d\geq 0$, $\ev_d$ is a linear map of
$K$-vector spaces. Its image is denoted by $C_X(d)$.
\begin{definition}\label{definition: evaluation code}
The \emph{evaluation code of order $d$} 
associated to $X$ is the linear subspace of $K^{\ls{X}}$ given by
$C_X(d)$, for $d\geq 0$. 
\end{definition}

Notice that if $q=2$ then $\TT^{s-1}$ is a point and, accordingly,
$C_X(d)=K$, for all $d$.  
For this reason, throughout this
article we assume that $q>2$. 

Clearly an evaluation code is a linear code, {\it{i.e.}}, it is a
linear subspace of $K^{\ls{X}}$. 
Accordingly, one defines the \emph{dimension} of the 
code as its dimension as a vector space, {\it{i.e.}}, as $\dim_K
C_X(d)$, its \emph{length} 
as the dimension of the ambient vector space, which, for evaluation
codes, coincides with  
$\ls{X}$ and, finally, its \emph{minimum distance}, is defined as: 
\[
\delta_X(d)=\min\{\|\ww\|
\colon 0\neq \ww\in C_X(d)\},
\]
where $\|\ww\|$ is the number of nonzero
coordinates of $\ww$. The basic parameters of $C_X(d)$ are related by the
{\it Singleton bound\/} for the minimum distance:
\[
\delta_X(d)\leq |X|-\dim_K C_X(d)+1.
\]

Two of the basic parameters of $C_X(d)$, the dimension and the
length, can be expressed 
using the Hilbert function of the quotient of $S$ by a particular homogeneous 
ideal. This ideal is the \emph{vanishing ideal} of $X$, {\it{i.e.}}, the
ideal of $S$ generated by the homogeneous polynomials of $S$ that
vanish on $X$. Denote it by $I(X)$. Recall that the \emph{Hilbert function} of
$S/I(X)$ is given by
\[
H_X(d):=\dim_K (S/I(X))_d=\dim_K S_d/I(X)_d=\dim_K C_X(d), 
\]
see \cite{Sta1}. The unique polynomial $h_X(t)=\sum_{i=0}^{k-1}c_it^i\in
\mathbb{Q}[t]$ of degree \mbox{$k-1=\dim S/I(X)-1$} such that $h_X(d)=H_X(d)$ for
$d\gg 0$ is called the {\it Hilbert polynomial\/} of $S/I(X)$. The
\mbox{integer} \mbox{$c_{k-1}(k-1)!$}, denoted by $\deg S/I(X)$, is
called the {\it degree\/} or  {\it multiplicity} of $S/I(X)$. In our
\mbox{situation} $h_X(t)$ is a nonzero constant because $S/I(X)$ has dimension $1$.
Furthermore \mbox{$h_X(d)=|X|$} for $d\geq |X|-1$, see \cite[Lecture
13]{harris} and \cite{geramita-cayley-bacharach}.
This means that $|X|$ is equal to the {\it degree\/}
of $S/I(X)$.

A good parameterized code should have large $|X|$ together with 
$\dim_K C_X(d)/|X|$ and $\delta_X(d)/|X|$ as large as possible.
Here, another algebraic invariant gives an indication of where to look for
nontrivial evaluation codes.

\begin{definition}\label{definition: regularity}
The \emph{index of regularity} of $S/I(X)$, denoted by
$\reg S/I(X)$, is the least integer $\ell\geq 0$ such that
$h_X(d)=H_X(d)$ for $d\geq \ell$. 
\end{definition}

As $S/I(X)$ is a $1$-dimensional Cohen-Macaulay graded algebra
\cite{geramita-cayley-bacharach}, the index of
regularity of $S/I(X)$ is the 
Castelnuovo-Mumford regularity of $S/I(X)$
\cite{eisenbud-syzygies}. We will refer to ${\rm reg}(S/I(X))$ 
simply as the {\it regularity\/} of $S/I(X)$. The regularity is
related to the degrees of a minimal generating set of $I(X)$.

\begin{definition} Let 
$f_1,\ldots,f_r$ be a minimal homogeneous generating set of $I(X)$.  
The {\it big degree\/} of $I(X)$ is defined as ${\rm bigdeg}\, I(X)=
\max_i\{\deg(f_i)\}$.  
\end{definition}

From the definition of the Castelnuovo-Mumford regularity of 
$S/I(X)$ \cite{eisenbud-syzygies}, one has: 

\begin{proposition}\label{castelnuovo-vs-bigdegree} 
${\rm bigdeg}\, I(X)-1\leq{\rm reg}(S/I(X))$.
\end{proposition}

Since $\dim_K C_X(d) = H_X(d)$ and the Hilbert polynomial of $S/I(X)$
is a constant polynomial  
with constant term equal to the dimension of the ambient vector
space, 
$K^{\ls{X}}$, we deduce that for 
$d\geq \reg S/I(X)$ the linear code $C_X(d)$ coincides with $K^{\ls{X}}$. This can also be
expressed by $\delta_X(d)=1$ for all $d\geq \reg S/I(X)$. We conclude that the potentially 
good codes $C_X(d)$ can occur only if $1\leq d < \reg(S/I(X))$. 

For a particular class of evaluation codes, called
\emph{parameterized linear codes}, 
the ideal $I(X)$ has been studied to an extent that it is possible to
use algebraic methods,  
based on elimination theory and Gr\"obner bases, to compute the
dimension and the length  
of $C_X(d)$, see \cite{algcodes}. Let us briefly describe the notion
of a parameterized linear code. 

Given an $n$-tuple of integers, $\nu=(r_1,\dots,r_n)\in \ZZ^n$, and a vector 
$\xx=(x_1,\dots,x_n)\in (K^*)^n$, 
we set $\xx^\nu = x_1^{r_1}\cdots x_n^{r_n}\in K^*$. 
Let $\nu_1,\dots,\nu_s\in\ZZ^n$ and let $X^*\subset (K^*)^s$ be the set given by:
\[
X^* = \set{(\xx^{\nu_1},\dots,\xx^{\nu_s}):\xx\in (K^*)^n}.
\]
Consider the multiplicative group structure of $(K^*)^s$ and let $\pi
\colon (K^*)^s \rt \TT^{s-1}$  
be the quotient map by the {\it diagonal subgroup\/} 
$\Lambda=\set{(\lambda,\dots,\lambda)\in (K^*)^s : \lambda \in K^*}$. Notice 
that $\TT^{s-1} = (K^*)^s/\Lambda$ is the projective torus in $\PP^{s-1}$.
\begin{definition}[\cite{afinetv},\cite{algcodes}]\label{definition:
parameterized linear code} 
Let $\nu_1,\dots,\nu_s\in \NN^n$. The set of points given by $X =
\pi(X^*)$ is called an  
\emph{algebraic toric set} parameterized by $\nu_1,\dots,\nu_s\in \NN^n$. 
The evaluation codes $C_X(d)$ obtained from an algebraic toric set $X$ are called 
\emph{parameterized linear codes}.  
\end{definition}

It is clear that $X^*$ is a subgroup of $(K^*)^s$, since it is the
image of the group homomorphism  
$(K^*)^n \rt (K^*)^s$ given by $\xx\mapsto (\xx^{\nu_1},\dots,\xx^{\nu_s})$. Denote by 
$\theta \colon (K^*)^n \rt X^*$ and by $\widetilde{\pi} \colon X^* \rt X$ the restrictions 
of the corresponding homomorphisms. Thus, we have the following sequence:
\begin{equation}\label{eq: structural exact sequence}
(K^*)^n \stackrel{\theta}{\lrt} X^* \stackrel{\widetilde{\pi}}{\lrt} X \lrt 1.  
\end{equation}

For a parameterized algebraic toric set $X$, the vanishing ideal
$I(X)$ carries extra structure.  
We know that, in this situation, $I(X)$ is $1$-dimensional
Cohen-Macaulay lattice  
ideal \cite{algcodes}. In particular $I(X)$ is a binomial ideal,
{\it{i.e.}}, it is generated 
by binomials. Recall that a binomial in $S$ is a polynomial of the
form $t^a-t^b$, where 
$a,b\in \NN^s$ and where, if \mbox{$a=(a_1,\dots,a_s)\in\NN^s$}, we set 
\[
t^a=t_1^{a_1}\cdots t_s^{a_s}\in S. 
\]
A binomial of the form $t^a-t^b$ is usually referred 
to as a {\it pure binomial\/} \cite{EisStu}, although here we are dropping the
adjective ``pure''.  

Let $\G$ be a simple graph with vertex set 
$V_\G=\set{v_1,\dots,v_n}$ and edge set $E_\G=\set{e_1,\dots,e_s}$.
Throughout the remainder of this article, when dealing with a graph, 
we shall reserve the use of $n$ and $s$ for 
the number of vertices and the number of edges of the graph in question. 
For an edge $e_i=\set{v_j,v_k}$, where $v_j,v_k\in V_\G$, let $\nu_i
= \mathbf{e}_j+\mathbf{e}_k\in \NN^n$, 
where, for $1\leq j\leq n$, $\mathbf{e}_j$ is the $j$-th element of
the canonical basis of $\QQ^n$. 

\begin{definition}[\cite{GR}]\label{definition: parameterized code associated to a graph}
The \emph{algebraic toric set associated to $\G$} is the toric set
parameterized by the $n$-tuples 
\mbox{$\nu_1,\dots,\nu_s\in \NN^n$}, obtained from the edges of $\G$.
If $X$ is the parameterized 
toric set associated to $\G$ we call its associated linear code
$C_X(d)$ \emph{the parameterized code 
associated to $\G$} and we refer to the vanishing ideal of $X$ as the
\emph{vanishing ideal over $\G$}. 
\end{definition}

If $\xx=(x_1,\dots,x_n) \in (K^*)^n$ and $e_i=\set{v_j,v_k}$ is an edge of $\G$, we set 
$\xx^{e_i}=\xx^{\mathbf{e}_j + \mathbf{e}_k}=x_jx_k$, so that the structural map 
$\theta\colon (K^*)^n \rt X^*$ is given by $\xx \mapsto (\xx^{e_1},\dots,\xx^{e_s})$.
It is clear that if $\G$ contains isolated vertices, then the associated algebraic toric set 
$X$ coincides with the algebraic toric set 
associated to the subgraph of $\G$ obtained by removing these
vertices. If $\G$ has a second edge through two 
vertices, then $X$ is isomorphic to its projection away from the
coordinate point of $\PP^{s-1}$  
corresponding to that edge; which, in turn, coincides with the
algebraic toric set defined by the graph  
obtained from $\G$ by removing the multiple edge. Hence, from the
point of view of the algebraic toric set  
$X$, the existence of multiple edges in $\G$ is not interesting.
If $\G$ has only one edge then is easy to see that $X=\PP^{s-1}$ is a point, 
$I(X)=0$ and $C_X(d)=K^*$. 
Thus throughout the remainder of this article we shall 
assume that $\G$ is a simple graph with no isolated vertices and with $s\geq 2$.

If $\G$ is a connected graph, the length of $C_X(d)$ has been determined.

\begin{proposition}[\rm{\cite[Corollary~3.8]{algcodes}}]\label{proposition:
order of X for G connected} 
Let $\G$ be a connected graph and $X$ its associated algebraic toric
set. 
Then $\ls{X}=(q-1)^{n-1}$ if $\G$ is non-bipartite and
$\ls{X}=(q-1)^{n-2}$ if $\G$ is bipartite. 
\end{proposition}

In particular, since $X\subset \TT^{s-1}\subset \PP^{s-1}$ and
$\ls{\TT^{s-1}}=(q-1)^{s-1}$ we see that if $\G$ is a connected
non-bipartite graph with $n=s$, then the  
algebraic toric set parameterized by the edges of $\G$ coincides with
$\TT^{s-1}$. 
In this situation, the vanishing ideal of $\TT^{s-1}$, its invariants
and all of the parameters  
of $C_X(d)$ are known, and are summarized
in the following proposition.

\begin{proposition}{\rm(\cite[Theorem~1,
Lemma~1]{GRH}, \cite[Corollary~2.2, Theorem~3.5]{ci-codes})}
If $\mathbb{T}^{s-1}$ is the projective torus
in $\mathbb{P}^{s-1}$, then
\renewcommand{\labelenumi}{(\roman{enumi})}
\begin{enumerate}\renewcommand{\itemsep}{.1cm}
\item $I(\mathbb{T}^{s-1})=\bigr(\{t_i^{q-1}-t_1^{q-1}\}_{i=2}^s\bigl)$;
\item $ F_{\mathbb{T}^{s-1}}(t)=(1-t^{q-1})^{s-1}/(1-t)^s$;
\item ${\rm reg}(S/I(\mathbb{T}^{s-1}))=(s-1)(q-2)$ and ${\rm
deg}(S/I(\mathbb{T}^{s-1}))=\ls{\TT^{s-1}}=(q-1)^{s-1}$;
\item $\dim_K C_{\TT^{s-1}}(d)=\sum_{j=0}^{\left \lfloor{d}/{(q-1)}
\right \rfloor}(-1)^j\binom{s-1}{j}\binom{s-1+d-j(q-1)}{s-1}$; 
\item $\delta_{\TT^{s-1}}(d) = (q-1)^{s-(k+2)}(q-1-\ell)$ for all
$d<{\rm reg}(S/I(\mathbb{T}^{s-1}))$, where $k\geq 0$ and $1\leq
\ell\leq q-2$ are the unique integers  
such that $d=k(q-2)+\ell$. 
\end{enumerate}
\end{proposition}

In the statement of the result, $F_{\TT^{s-1}} (t) =
\sum_{i=0}^\infty H_{\TT^{s-1}} (i) t^i$ is  
the Hilbert series of $S/I(\TT^{s-1})$. The fact that
the vanishing ideal in the case of the torus is a complete intersection plays a crucial part
in the proof of these results. We know that in practice the vanishing ideal associated to a
general graph is far from being a complete intersection. Indeed, by
\cite[Corollary 4.5]{ci-codes}  
for an algebraic toric set $X$ associated to a graph (or more
generally a \emph{clutter}---see \cite{ci-codes}  
for a definition), $I(X)$ is a complete intersection if and only if $X=\TT^{s-1}$.

\section{The length of parameterized codes of graphs}\label{length}

We continue to use the notation and definitions used in
Section~\ref{sec: prelimiaries}. In this
section, we show an explicit formula for the length of any
parameterized code associated to an arbitrary graph. 

Let $\G$ be a simple graph with vertex set
$V_\G=\set{v_1,\dots,v_n}$ and edge set 
$E_\G=\set{e_1,\dots,e_s}$. 
Denote by $\G_1,\dots,\G_m$ the connected
components of $\mathcal{G}$.
For each $1\leq j\leq m$, let $n_j$ and $s_j$ denote the number of
vertices and edges of $\G_j$, respectively; 
so that $n=n_1+\cdots + n_m$ and 
$s=s_1+\cdots + s_m$. Denote the edges of $\G_j$ by $\set{e_{j1},\dots,e_{js_j}}$, 
let $X_j\subset \PP^{s_j-1}$ be the algebraic toric set parameterized by $\G_j$ and let
\[
(K^*)^{n_j} \stackrel{\theta_j}{\lrt} X^*_j \stackrel{\widetilde{\pi}_j}{\lrt} X_j \lrt 1 
\]
be the corresponding structural sequences. Since for fixed distinct $j_1\not = j_2$ the edges 
$e_{j_1k_1}$ and $e_{j_2k_2}$ have no vertex in common and thus $\xx^{e_{j_1k_1}}$ and 
$\xx^{e_{j_2k_2}}$ involve disjoint sets of coordinates of the vector $\xx$, we deduce
that $\theta \colon (K^*)^n \rt X^*$ is isomorphic to 
\[
\theta_1\times\cdots \times  \theta_m \colon (K^*)^{n_1}\times \cdots
\times (K^*)^{n_m} \rt X^*_1\times \cdots  \times X^*_m. 
\]
In particular $\ls{X^*}=\prod_{j=1}^m \smls{X^*_j}$. We need to find the 
order of the kernel of the maps $\widetilde{\pi}_j$.

\begin{lemma}\label{lemma: the order of the kernel of pi}
Let $\G$ be a connected graph. If $\G$ is non-bipartite, then
$\ls{\Ker \widetilde{\pi}}=\frac{q-1}{2}$ if $q$ is odd and $\ls{\Ker
\widetilde{\pi}}=q-1$ if $q$ is even. If $\G$ is bipartite, then $\ls{\Ker \widetilde{\pi}}=q-1$.
 \end{lemma}

\begin{proof}
Let $\xx\in (K^*)^{n}$. Then 
\mbox{$\theta(\xx)=(1,\dots,1)$} implies that $\xx^e=1$ for all $e\in E_{\G}$.
Suppose $\G$ is non-bipartite. Then $\G$ contains an odd cycle. 
We assume, without loss of generality, that the
edges in this cycle are 
\[
e_1=\set{v_1,v_2},\dots,e_{2k-1}=\set{v_{2k-1},v_1},
\]
where $v_1\dots,v_{2k-1}\in V_{\G}$.
We deduce that $x_1x_2=\cdots=x_{2k-1}x_1=1$, which, in turn, 
implies that \mbox{$x_1=\cdots =x_{2k-1}=u\in  K^*$} with $u^2=1$. 
Now, let $v_r\in V_{\G}$ be any vertex of $\G$. Then, there exists a path 
$$\set{v_1,v_{\ell_1}},\set{v_{\ell_1},v_{\ell_2}},\dots,\set{v_{\ell_k},v_r}$$
connecting $x_1$ with $x_r$. Since
$x_1x_{j_1}=x_{j_1}x_{j_2}=\cdots=x_{j_k}x_r=1$, we deduce 
that $x_r=u$. Hence, either $\xx=(1,\dots,1)$ or $\xx=(-1,\dots,-1)$, from which we 
conclude that $\ls{\Ker\theta}=2$ if $q$ is odd and
$\ls{\Ker\theta}=1$ if $q$ even. 
Suppose now that $\G$ is bipartite, and, without loss of generality, 
denote the bipartition of $V_\G$ by $\set{v_1,\dots,v_\ell}\cup \set{v_{\ell+1},\dots,v_n}$. 
Let $v_r$ be any vertex and let 
\[
\set{v_1,v_{j_1}},\set{v_{j_1},v_{j_2}},\ldots,\set{v_{j_k},v_r}
\]
be a path connecting $v_1$ with $v_r$.
Notice that $\set{v_{j_1},v_{j_3},\dots}$ is a subset of $\set{v_{\ell+1},\dots,v_n}$ and 
$\set{v_{j_2},v_{j_4},\dots}$ is a subset of $\set{v_1,\dots,v_\ell}$. 
From $x_1x_{j_1}=x_{j_1}x_{j_2}=\cdots = x_{j_k}x_r=1$ we deduce that $x_r=x_1$ if 
$v_r\in \set{v_1,\dots,v_\ell}$ or $x_r=x_1^{-1}$ otherwise. 
Hence $x=(x_1,\dots,x_1,x_1^{-1},\dots,x_1^{-1})$, {\it{i.e.}},
the  $\ell$ first coordinates of $\xx$ are equal to $x_1$ and the remaining ones are equal to $x_1^{-1}$. 
Conversely, it is easy to see that 
any element of $(K^*)^n$ of the form $(u,\dots,u,u^{-1},\dots,u^{-1})$ belongs to $\Ker \theta$. 
We deduce that in this case $\ls{\Ker \theta}=q-1$. The proof now follows easily 
from Proposition~\ref{proposition: order of X for G connected}.
Indeed, we know that the order of $X$ is $(q-1)^{n-1}$,  
if $\G$ is non-bipartite and $(q-1)^{n-2}$ otherwise. Hence, 
$\ls{\Ker \widetilde{\pi}}=\frac{q-1}{2}$ if
$\G$ is non-bipartite and $q$ is odd, $\ls{\Ker \widetilde{\pi}}={q-1}$ if
$\G$ is non-bipartite and $q$ is even, and $\ls{\Ker
\widetilde{\pi}}={q-1}$ if $\G$ is bipartite. 
\end{proof}

We come to the main result of this section. 

\begin{theorem}\label{theorem: computation of the length for any graph}
Suppose $\G$ has $m$ connected components, of which $\gamma$ are non-bipartite. Then,
\[\renewcommand{\arraystretch}{1.3}
\ls{X} = \left \{ 
\begin{array}{l}
\bigr(\frac{1}{2}\bigl)^{\gamma-1}(q-1)^{n-m+\gamma-1}{}, \text{ if }
\gamma\geq 1\text{ and }q \text{ is odd}, \\
(q-1)^{n-m+\gamma-1}{}, \text{ if }
\gamma\geq 1\text{ and }q \text{ is even}, \\
(q-1)^{n-m-1}, \text{ if }\gamma=0. 
\end{array}
\right.
\]
\end{theorem}

\begin{proof} As in the discussion above, let $X_1,\dots,X_m$ be the
parameterized toric sets associated to  
the connected components of $\G$. Then $\ls{X^*}=\prod_{j=1}^m \smls{X^*_j}$, which, by 
Lemma~\ref{lemma: the order of the kernel of pi}, is given by
\[\renewcommand{\arraystretch}{1.3}
\ls{X^*} = \left \{ 
\begin{array}{l}
\bigr(\frac{1}{2}\bigl)^{\gamma}(q-1)^{n-m+\gamma}{}, \text{ if } q \text{ is odd}, \\
(q-1)^{n-m+\gamma}{}, \text{ if } q \text{ is even}.
\end{array}
\right.
\]
From the proof of
Lemma~\ref{lemma: the order of the kernel of pi}, it is seen that
the kernel of the map $\widetilde{\pi}\colon X^*\rightarrow X$ is
equal to $\Lambda$, the diagonal subgroup of $(K^*)^s$, if $\gamma=0$,
 and it is equal to $\Lambda^2=\{(\lambda^2,\ldots,\lambda^2)\vert\, \lambda\in
F_q^*\}$ if $\gamma\geq 1$. The subgroup $\Lambda$ has order $q-1$.
The subgroup $\Lambda^2$ has order $q-1$ if $q$ is even and has order
$(q-1)/2$ if $q$ is odd (this follows readily using the map
$\lambda\mapsto (\lambda^2,\ldots,\lambda^2)$). As $|X|=|X^*|/|{\rm
Ker}\, \widetilde{\pi}|$, the result follows. 
\end{proof}

\begin{example} Let $G$ be the graph whose connected components are 
a triangle and a square. Thus, $n=7$, $m=2$, $\gamma=1$. Using the
formula of Theorem~\ref{theorem: computation of the length for any
graph}, we get: (a) $|X|=1024$ if $q=5$, and (b) $|X|=243$ if
$q=2^2$.
\end{example}

\section{Degree bounds for the generators of $I(X)$}\label{sec: generators of I(X)}

We continue to use the notation and definitions used in
Section~\ref{sec: prelimiaries}. In this section $X\subset \PP^{s-1}$ is the algebraic 
toric set parameterized by $\nu_1,\ldots,\nu_s\in\mathbb{N}^n$ and 
$I(X)\subset S=K[t_1,\dots,t_s]$ is the vanishing ideal of $X$.
We show some degree bounds for a minimal set of generators of $I(X)$ consisting of
binomials.

Recall that by \cite{algcodes} we know that $I(X)$ is generated by homogeneous 
binomials $t^a-t^b$, with $a,b\in \NN^s$. 
There are a number of elementary observations to be made. Let
$f=t^a-t^b$ be a nonzero binomial of $S$. Firstly, since $X\subset \TT^{s-1}$, evidently
$I(\TT^{s-1})\subset I(X)$, hence $t_i^{q-1}-t_j^{q-1}\in I(X)$, for all $1\leq i,j\leq s$. 
Secondly, if $\gcd(t^a,t^b)\not = 1$, then
we can factor the greatest common divisor $t^c$ from both $t^a$ and $t^b$ to obtain 
\mbox{$t^a-t^b = t^c(t^{a'}-t^{b'})$}, 
for some $a',b'\in \NN^s$. Since $t^c$ is never zero on $\TT^{s-1}$, for any $c\in \NN^s$, we deduce that
\mbox{$t^a-t^b \in I(X)$} if and only if $t^{a'}-t^{b'}\in I(X)$.
Therefore, when looking for ``binomial generators'' 
of $I(X)$ we may \mbox{restrict} ourselves to those binomials 
$t^a-t^b$ such that $t^a$ and $t^b$ have no common divisors.
Given $a=(a_1,\dots,a_s)\in \NN^s$, we set $|a|= a_1+\cdots + a_s$ and  \mbox{$\supp(a)=\set{i : a_i\not = 0}$}. Then, clearly, 
$t^a$ and $t^b$ have no common divisors if and only if $\supp(a)\cap \supp(b)=\emptyset$.

\begin{definition}\rm A subgroup of $\mathbb{Z}^s$ is called a {\it
lattice\/}. A {\it lattice ideal\/} is an 
ideal of the form
$$
I(\mathcal{L})=(\{t^{a}-t^{b}\, \colon\, 
a-b\in\mathcal{L}\mbox{ and }\supp(a)\cap \supp(b)=\emptyset\})\subset S
$$
for some lattice $\mathcal{L}$ of $\mathbb{Z}^s$. 
\end{definition}

\begin{lemma}\label{l-lemma} Let $L\subset S$ be a lattice ideal generated by 
$\mathcal{B}=\{t^{a_i}-t^{b_i}\}_{i=1}^r$. Then,  
{\rm (a)} $L=I(\mathcal{L})$, where $\mathcal{L}$ is the subgroup of $\mathbb{Z}^s$ 
generated by $\{a_i-b_i\}_{i=1}^r$, and {\rm (b)} if 
$t^{a_i}-t^{b_i}$ is homogeneous for all $i$ and $f=t^a-t^b\in L$,
then $f$ is homogeneous.
\end{lemma}

\begin{proof} Part (a) follows from \cite[Lemma~7.6]{cca}. To show
(b) notice that, from part (a), $f\in I(\mathcal{L})$. Then, $a-b$ is a linear combination of
$\{a_i-b_i\}_{i=1}^r$. Thus, if $\mathbf{1}=(1,\ldots,1)$, we get 
that $|a|-|b|$ is equal to $\langle\mathbf{1},a-b\rangle=0$ because
$\langle\mathbf{1},a_i-b_i\rangle=0$ for all $i$. Thus, $|a|=\deg(t^a)=\deg(t^b)=|b|$.
\end{proof}

\begin{lemma}\label{h-lemma} If $f=t^a-t^b\in I(X)$, then $f$ is homogeneous. 
\end{lemma}

\begin{proof} According to \cite[Theorem~2.1]{algcodes}, $I(X)$ is lattice ideal 
generated by homogeneous binomials. Thus, the lemma follows from
Lemma~\ref{l-lemma}.
\end{proof}

\begin{lemma}\label{lemma: reduction of the degree in each variable}
Let $f=t^a-t^b \in I(X)$, where $a,b \in \NN^s$ and $\supp(a)\cap \supp(b)=\emptyset$.
Suppose that there exists 
$i$ such that $t_i^{q-1}$ divides $t^a$ and $\supp(b)\not = \emptyset$. 
Then, there exists a binomial $g \in I(X)$, with $\deg(g)<\deg(f)$, 
and there exists $j$, such that $f-t_{j}g \in I(\TT^{s-1})$. 
\end{lemma}

\begin{proof}
Write $t^a=t_i^{q-1}t^{a'}$, with $a'\in \NN^s$. 
Since $\supp(b)\not =\emptyset$, there exists $j$
such that $t_j$ divides $t^b$. Write $t^b=t_jt^{b'}$, for some $b'\in \NN^s$.
Then,
\[
t^a-t^b= t_i^{q-1}t^{a'}-t_j t^{b'} = (t_i^{q-1}-t_j^{q-1})t^{a'} + t_j(t_j^{q-2}t^{a'}-t^{b'}).
\]
Set $g=t_j^{q-2}t^{a'}-t^{b'}$. Then, since $t_i^{q-1}-t_j^{q-1}\in I(X)$, we see that 
$g\in I(X)$ and, moreover, it is clear that if $g\not = 0$ then
\mbox{$\deg(g)=\deg(f)-1$}. 
\end{proof}

\begin{theorem}\label{lemma: first reduction}
There exists a set of generators of $I(X)$ which consists of  
the toric relations $t_i^{q-1}-t_j^{q-1}$ plus 
a finite set of homogeneous binomials $t^a-t^b$ with $\supp(a)\cap \supp(b)= \emptyset$ and such that 
the degree of $t^a-t^b$ in each of the variables $t_i$ is $\leq q-2$.
\end{theorem}

\begin{proof} We know that $I(X)$ is generated by 
binomials \cite{algcodes}. If $\set{f_1,\dots,f_r}$ 
is a set of binomials 
generating $I(X)$, then so is the set 
$$\mathcal{B}=\set{f_1,\dots,f_r}\cup \{t^{q-1}_i-t^{q-1}_j : 
1\leq i,j\leq s\}.$$ 
If $f_i\in I(\mathbb{T}^{s-1})$, we have
$(\mathcal{B})=(\{\set{f_1,\dots,f_r}\setminus\{f_i\}\}\cup \{t^{q-1}_i-t^{q-1}_j :  
1\leq i,j\leq s\})$. Thus, we may assume that $\mathcal{B}$ is a
generating set of $I(X)$ with $f_i\notin I(\mathbb{T}^{s-1})$ for all
$i$. By the discussion above we may also assume that each $f_i$ is of the form
$t^a-t^b$ with $\supp(a)\cap\supp(b)=\emptyset$. We can write $f_1=t^a-t^b$, with $a,b\in\NN^s$.
Suppose that there exists $i$ such that $t_i^{q-1}$ divides $t^a$ or
$t^b$. Hence, since $f_1$ is homogeneous by Lemma~\ref{h-lemma}, we
deduce that the sets $\supp(a)$ and $\supp(b)$  
are both nonempty. Then, from Lemma~\ref{lemma: reduction of the degree in each variable}, 
there exists $j$ and a homogeneous binomial $g_1'\in I(X)$ such that
$\deg(g_1')<\deg(f_1)$ and $f_1-t_jg_1'\in I(\TT^{s-1})$. We can 
write $g_1'=t^cg_1$ for some $c\in\mathbb{N}^s$, where $g_1$ is a
binomial in $I(X)$ whose terms have disjoint support. Clearly,
\[
I(X)=(\mathcal{B})=\bigl(\set{g_1,f_2,\dots,f_r}\cup \{t^{q-1}_i-t^{q-1}_j : 
1\leq i,j\leq s\}\bigr)
\]
and $g_1\notin I(\TT^{s-1})$. If there exists $i$ such that
$t_i^{q-1}$ divides one of the terms of $g_1$, we repeat the previous
procedure with $g_1$ playing the role of $f_1$ and obtain a binomial
$g_2$, and so on. Thus, by iterating the previous procedure, we obtain
a sequence of homogeneous binomials $f_1,g_1,\dots,g_m$, with decreasing
degrees, 
such that
\begin{equation}\label{nov29-11}
I(X)=(\mathcal{B})=\bigl(\set{g_m,f_2,\dots,f_r}\cup \{t^{q-1}_i-t^{q-1}_j : 
1\leq i,j\leq s\}\bigr)
\end{equation}
and $g_m\notin I(\TT^{s-1})$. Thus, using the previous procedure 
enough times, we obtain a binomial $g_m=t^{a'}-t^{b'}$ none of whose
 terms $t^{a'}$ or $t^{b'}$ is divisible by any $t_i^{q-1}$, for
 $1\leq i\leq s$. If we
proceed in this manner, with each of the remaining  $f_2,\dots,f_r$, 
we reach a generating set satisfying the condition in the statement.
\end{proof}

The next proposition is intended mainly for practical applications.
It gives a bound on the degrees of a minimal set of generators of
$I(X)$. It is a valuable tool to use 
when implementing the calculation of $I(X)$ in a computer algebra software.

\begin{proposition}\label{proposition: bound on the degrees of generators}
Set $k=\left \lfloor \frac{s}{2} \right\rfloor$. If $k\geq 2$, then the vanishing ideal of $X$ has
a generating set whose elements have degree $\leq k(q-2)$.
\end{proposition}

\begin{proof}
Let $t^a-t^b\in I(X)$ be a homogeneous binomial. Write $a=(a_1,\dots,a_s)\in \NN^s$ and
$b=(b_1,\dots,b_s)\in \NN^s$. By Theorem~\ref{lemma: first reduction}, 
we may assume that $\supp(a)\cap \supp(b)=\emptyset$ and that 
$0\leq a_i,b_j\leq q-2$. Let $r=\ls{\supp(a)}$ and $\ell=\ls{\supp(b)}$. 
Then, either $r$ or $\ell$ is $\leq k$, for otherwise:
\[
r+\ell\geq 2k+2 =  2\left \lfloor {s}/{2} \right\rfloor +2 \geq s+1,
\]
which is impossible. Assume $r\leq k$. Then,
$\deg(t^a-t^b)=a_1+\cdots +a_s \leq r(q-2)\leq k(q-2)$.
\end{proof}

If $X$ is the algebraic toric set associated to a cycle $\G$ of order
$s=2k$, then, by 
Corollary~\ref{remark: on the degrees of the generators of I(X)}, $I(X)$
is generated in degrees $\leq (k-1)(q-2)+1$.  
Hence for this restricted class of vanishing ideals our estimate
is not sharp. 
On the other hand, for $q=3$, the estimate that $I(X)$ is generated in degrees
$\leq k$ is sharp, as the following example shows.

\begin{small}
\begin{figure}[h]
\begin{picture}(200,70)(4,20)
\put(20,80){$\bullet$}
\put(10,85){$1$}
\put(7,50){$e_1$}
\put(22.5,82){\line(0,-1){60}}
\put(20,20){$\bullet$}
\put(10,15){$2$}
\put(47,28){$e_2$}
\put(22.2,22.2){\line(5,3){50}}
\put(70,50){$\bullet$}
\put(80,50){$3$}
\put(47,72){$e_3$}
\put(72.2,52.7){\line(-5,3){50}}
\put(130,50){$\bullet$}
\put(120,50){$4$}
\put(147,72){$e_4$}
\put(132.2,52.2){\line(5,3){50}}
\put(180,80){$\bullet$}
\put(190,85){$5$}
\put(182.5,82){\line(0,-1){60}}
\put(190,50){$e_5$}
\put(180,20){$\bullet$}
\put(190,15){$6$}
\put(147,28){$e_6$}
\put(184,22){\line(-5,3){50}}
\end{picture}
\caption{}
\label{fig: graph with large degree generators}
\end{figure}
\end{small}

\begin{example}\label{example: two triangles}
Let $\G$ be the graph in Figure~\ref{fig: graph with large degree generators} and assume that $q=3$.
Then, using \emph{Macaulay}$2$ \cite{mac2}, 
we found that $I(X)$ is generated by the (minimal) set of binomials:
\[
\begin{array}{c}
t_5^2-t_6^2,\quad t_4^2-t_6^2,\quad t_3^2-t_6^2,\quad t_2^2-t_6^2,\quad t_1^2-t_6^2,\\
t_3t_4t_5-t_1t_2t_6, \quad t_2t_4t_5-t_1t_3t_6,\quad
t_1t_4t_5-t_2t_3t_6,\quad t_2t_3t_5-t_1t_4t_6,\quad t_1t_3t_5-t_2t_4t_6,\\ 
t_1t_2t_5-t_3t_4t_6, \quad t_2t_3t_4-t_1t_5t_6,\quad
t_1t_3t_4-t_2t_5t_6,\quad t_1t_2t_4-t_3t_5t_6,\quad t_1t_2t_3-t_4t_5t_6. 
\end{array}
\]
\end{example}

\section{Generators of $I(X)$ for even cycles and certain bipartite graphs}
\label{sec: generators of I(X) for
graphs}

We keep the notation of Section~\ref{length}: $X\subset \PP^{s-1}$ is the algebraic 
toric set parameterized by a graph $\G$ and $I(X)\subset S=K[t_1,\dots,t_s]$ is the vanishing ideal 
of $X$. This section is devoted to giving an explicit 
description of a binomial generating set for $I(X)$, when $\G=\C_{2k}$ is a
cycle of even order, or when $\G$ is a bipartite graph whose cycles are
pairwise vertex disjoint. 

\begin{proposition}\label{proposition: nonoccuring variables}
Let $f=t^a-t^b\in I(X)$, with $a=(a_1,\ldots,a_s)$ and
$b=(b_1,\ldots,b_s)$, such that \mbox{$\supp(a)\cap \supp(b)=\emptyset$} and 
$a_j,b_j\leq q-2$ for all $j$. {\rm (a)} If $G$ is a connected bipartite 
graph and $e_i$ is an edge of $\G$ which does not belong to any cycle
of $\G$, then $a_i=b_i=0$. {\rm (b)} If $\G$ is any graph 
and $\G$ has an edge $e_i$ with a degree $1$ incident
vertex, then $a_i=b_i=0$.
\end{proposition}

\begin{proof} (a) Assume, without loss of generality that,
$e_i=\set{v_1,v_2}$. In what follows we use the symbol $\sqcup$ to
denote a disjoint union of objects. Since $\G$ is bipartite there exist a 
bipartition $V_G=A\sqcup B$ with, say,
$v_1\in A$ and $v_2\in B$. Since $e_i$ does not belong to a cycle of $\G$, the removal
of edge $e_i$ produces a disconnected graph $\G_1\sqcup \G_2$, with $v_1\in V_{\G_1}$ and
$v_2\in V_{\G_2}$. Let $u\in K^*$ be a generator of the multiplicative
group of $K$. Let us label the vertices of $\G$ with one of the
elements $u$, $u^{-1}$ or $1$, 
according to the rule that we now explain. Let $v_r$ be any vertex. If
$v_r\in V_{\G_1}$ label
$v_r$ with $1$, if $v_r\in V_{\G_2}\cap A$ label $v_r$ with $u^{-1}$, and if
$v_r\in V_{\G_2}\cap B$ label $v_r$ with $u$. Consider
$\xx=(x_1,\dots,x_n)\in (K^*)^{n}$  where, for $1\leq r\leq n$, the coordinate  
$x_r$ takes on the value of the label of $v_r$. Then $\xx^{e_j}=1$ if $j\not = i$ and 
$\xx^{e_i}=u$.  Assume that $a_i>0$, then $b_i=0$  because $a$ and $b$
have disjoint support. Thus $f(\xx^{e_1},\dots,\xx^{e_s})=0$, implies that $u^{a_i}-1=0$, a
contradiction because $1\leq a_i\leq q-2$.  Similarly if $b_i>0$ we
derive a contradiction. Hence, we deduce that $a_i=b_i=0$. (b) This
part follows using a similar argument. 
\end{proof}

\begin{example}\label{example: two triangles joined at the hip}
For non-bipartite graphs Proposition~\ref{proposition:
nonoccuring variables}(a) does not hold. 
Let $\G$ be the graph in Figure~\ref{fig: graph -- two triangles
joined at the hip} and assume that $q=5$. 
Then, using \emph{Macaulay}$2$ \cite{mac2}, we found that the binomial   
$t_1t_2t_4^2t_7-t_3t_5^2t_6t_8$
is in a minimal generating set of $I(X)$. In this monomial the
variables $t_4$ and $t_5$, which are not 
in any cycle of $\G$, occur.
\end{example}

\begin{small}
\begin{figure}[h]
\begin{picture}(200,70)(33,20)
\put(20,80){$\bullet$}
\put(10,85){$1$}
\put(7,50){$e_3$}
\put(27,50){$\scriptstyle 1$}
\put(22.5,82){\line(0,-1){60}}
\put(20,20){$\bullet$}
\put(10,15){$3$}
\put(47,28){$e_2$}
\put(43,40){$\scriptstyle 1$}
\multiput(22.2,20.5)(2.5,1.5){21}{$\cdot$}
\put(70,50){$\bullet$}
\put(72,58){$2$}
\put(47,72){$e_1$}
\put(43,60){$\scriptstyle 1$}
\multiput(72.2,49)(-2.5,1.5){21}{$\cdot$}
\multiput(72.2,50.25)(3,0){21}{$\cdot$}
\put(98.5,57.5){$e_4$}
\put(101,42.5){$\scriptstyle 2$}
\put(130,50){$\bullet$}
\put(130,58){$4$}
\put(132,53.5){\line(1,0){60}}
\put(158,57.5){$e_5$}
\put(160,42.5){$\scriptstyle  2$}
\put(187.5,58){$5$}
\put(190,50){$\bullet$}
\put(207,72){$e_6$}
\put(218,60){$\scriptstyle 1$}
\put(192.2,52.2){\line(5,3){50}}
\put(240,80){$\bullet$}
\put(250,85){$6$}
\multiput(241,81)(0,-3){21}{$\cdot$}
\put(250,50){$e_7$}
\put(234,50){$\scriptstyle 1$}
\put(240,20){$\bullet$}
\put(250,15){$7$}
\put(207,28){$e_8$}
\put(218,40){$\scriptstyle  1$}
\put(244,22){\line(-5,3){50}}
\end{picture}
\caption{}
\label{fig: graph -- two triangles joined at the hip}
\end{figure}
\end{small}

\begin{corollary}\label{corollary: supports add up to all}
Suppose that $\G=\C_{2k}$ is a cycle of even order. 
Let $f=t^a-t^b$ be a nonzero homogeneous binomial in $I(X)$, 
with $a=(a_1,\dots,a_s)\in \NN^s$ and $b=(b_1,\dots,b_s)\in \NN^s$ 
such that \mbox{$\supp(a)\cap \supp(b)=\emptyset$} and $0\leq a_i,b_j\leq q-2$. 
Then $\supp(a)\cup\supp(b)=\set{1,\dots,s}$.
\end{corollary}

\begin{proof}
Assume, without loss of generality that $s\not \in \supp(a)\cup
\supp(b)$. Then, 
$f$ is a polynomial in the variables $t_1,\dots,t_{s-1}$ which vanishes
along the projection of $X$ onto the first $s-1$ coordinates. 
The algebraic toric set obtained after projecting is none other that the algebraic toric 
set associated with the graph obtained from $\G=\C_{2k}$ by removing
the edge $e_s$, which 
is a tree. Hence, by Proposition~\ref{proposition: nonoccuring
variables}, none of the 
remaining variables $t_1,\dots,t_{s-1}$ occurs in $f$, in other words, $f=0$, which is a 
contradiction.
\end{proof}

From now on, until otherwise stated, we will restrict to the case of $\G=\C_{2k}$, 
a cycle of order $2k$ with $k\geq 2$. 
Let $V_{\C_{2k}}=\set{v_1,\dots,v_{2k}}$ and 
$e_i=\set{v_i,v_{i+1}}$ for $1\leq i \leq 2k-1$ and \mbox{$e_s=e_{2k}=\set{v_{2k},v_1}$}.
We are now ready to give a combinatorial description of the generators of
$I(X)$ other than those coming from the toric relations. From 
Theorem~\ref{lemma: first reduction} and Corollary~\ref{corollary: supports add up to all} we
know that there is a set of generators of $I(X)$ consisting of the toric generators 
$t_i^{q-1}-t_j^{q-1}$ plus a set of binomials of the type $t^a-t^b$ where
$a=(a_1,\dots,a_s)\in \NN^s$, $b=(b_1,\dots,b_s)\in \NN^s$ are such that 
\mbox{$\supp(a)\sqcup \supp(b) = \set{1,\dots,s}$} and $0\leq a_i,b_j\leq q-2$.
Hence to any such binomial one can associate a partition of $\set{1,\dots,s}$. 
For the remainder of this article, given $r\in
\set{1,\dots,q-2}$ we will
fix the following notation:
\[
\hat{r}=q-1-r.
\]
\begin{definition}\label{definition: recursive function} 
Let $\sigma=A\sqcup B$  be a partition of $\set{1,\dots,s}$ and
fix $r\in \set{1,\dots,q-2}$. Define a function 
\mbox{$\rho_\sigma^r\colon \set{1,\dots,s}\rt \set{r,\hat{r}}$},
recursively, by setting $\rho_\sigma^r(1)=r$ and, 
\begin{equation}\label{eq: recursive relation}
\begin{cases}
\rho_\sigma^r(i+1)=\widehat{\rho_\sigma^r(i)}, &  \text{if
$\set{i,i+1}\subset A$ or $\set{i,i+1}\subset B$}\\ 
\rho_\sigma^r(i+1)=\rho_\sigma^r(i), & \text{otherwise,}
\end{cases}
\end{equation}
for every $1\leq i\leq s-1$. 
\end{definition}

Notice that, for every $i\in
\set{1,\dots,s-2}$, 
$\rho_\sigma^r(i)=\rho_\sigma^r(i+2)$ if and only if
$i$ and $i+2$ are in the same partition. Since $s$ is even, we deduce that 
$\rho_\sigma^r(1)=\rho_\sigma^r(s-1)$ if and only if $1$ and $s-1$ are in the same partition.
This implies that $\rho_\sigma^r(1)$ can be defined from $\rho_\sigma^r(s)$ 
using the same recursive formula. Indeed, working in $\set{1,\dots,s}$ modulo $s$,
the function $\rho_\sigma^r$ can be recovered recursively, using the above rule, from 
$\rho^r_\sigma(k)$, for any $k\in \set{1,\dots,s}$. The following lemma will be used
in the proofs of some of the results below.

\begin{lemma}\label{lemma: technical lemma}
Let $\sigma=A\sqcup B$ be a partition of $\set{1,\dots,s}$ and $r\in
\set{1,\dots,q-2}$. 
Consider $i\in A$ and $\sigma'=A'\sqcup B'$ where $A'=A\setminus \set{i}$ and 
$B'=B\cup \set{i}$. Let 
$\rho\colon \set{1,\dots,s}\rt \set{r,\hat{r}}$ be given by $\rho(j)=\rho_\sigma^r(j)$ for every
$j\not=i$ and $\rho(i)=\widehat{\rho_\sigma^r(i)}$. Then $\rho=\rho_{\sigma'}^r$, if 
$i>1$ or $\rho=\rho_{\sigma'}^{\hat{r}}$, if $i=1$.
\end{lemma}

\begin{proof}
We will look first at the case $i=1$. In this case,
$\rho(1)=\widehat{\rho_\sigma^r(1)}=\hat{r}=\rho_{\sigma'}^{\hat{r}}(1)$. 
If $2 \in A$, then $\rho(2)=\rho_{\sigma}^r(2)=\hat{r}$, according to 
the definition of the function $\rho$, to Definition~\ref{definition: recursive function}
and to the fact that $1 \in A$. But if $2 \in A$, then $2 \in A'$, and 
$\rho_{\sigma'}^{\hat{r}}(2)=\hat{r}$ since $1 \in B'$. If $2 \in B$, then 
$\rho(2)=\rho_{\sigma}^r(2)=r$; but if $2 \in B$, then $2 \in B'$, and 
$\rho_{\sigma'}^{\hat{r}}(2)=r$.
In any case, $\rho(2)=\rho_{\sigma'}^{\hat{r}}(2)$.
Let $j \in \set{3,\dots,s}$. By definition, $\rho(j)=\rho_{\sigma}^r(j)$;
and $\rho_{\sigma}^r(j)$ is determined by $\sigma$, by $\rho_{\sigma}^r(2)$
and by Eq.~(\ref{eq: recursive relation}).
Since $\rho_{\sigma'}^{\hat{r}}(j)$ is determined by $\sigma'$, by $\rho_{\sigma'}^{\hat{r}}(2)$
and by Eq.~(\ref{eq: recursive relation}), 
since $\rho_{\sigma}^r(2)=\rho_{\sigma'}^{\hat{r}}(2)$, and since
the partitions $\sigma$ and $\sigma'$ agree in $\set{2,\dots,s}$,
%$j \in A$ if and only if $j \in A'$ and $j \in B$ if and only if $j \in B'$, 
we conclude that $\rho(j)=\rho_{\sigma}^r(j)=\rho_{\sigma'}^{\hat{r}}(j)$. Therefore,
$\rho=\rho_{\sigma'}^{\hat{r}}$. For the case $i>1$, we use a similar argument to show that 
$\rho=\rho_{\sigma'}^{r}$. 
\end{proof}

Given any $\sigma = A \sqcup B$, a partition of $\set{1,\dots,s}$,
if, without loss in generality, we choose $1 \in A$,
it is clear that given any 
$r\in\set{1,\dots,q-2}$, there exist unique 
$a$ and $b$ in $\NN^s$ such that $\supp(a)=A$,
$\supp(b)=B$, $a_i=\rho_\sigma^r(i)$ if $i\in \supp(a)$ and
$b_j=\rho_\sigma^r(j)$ if $j\in \supp(b)$.

\begin{definition}
Let $\sigma=A\sqcup B$ be a partition of $\set{1,\dots,s}$  with $1\in A$
and let $r\in\set{1,\dots,q-2}$. We denote
by $f_\sigma^r$ the unique binomial $t^a-t^b$ such that $\supp(a)=A$,
$\supp(b)=B$, $a_i=\rho_\sigma^r(i)$ if $i\in \supp(a)$ and
$b_j=\rho_\sigma^r(j)$ if $j\in \supp(b)$.
\end{definition}

\begin{figure}[h]
\begin{picture}(150,140)(-70,-65)

\put(-60,0){$\bullet$}
\put(-71,0){$v_1$}
\put(-56.5,2){\line(2,-5){17}}
\put(-45,-20){$6$}
\put(-60,-20){$e_1$}

\put(-42.4,-42.4){$\bullet$}
\put(-50,-48){$v_2$}
\multiput(-41.4,-42.4)(2.5,-1){18}{$\cdot$}
\put(-19,-45){$6$}
\put(-22,-60){$e_2$}

\put(0,-60){$\bullet$}
\put(0,-68){$v_3$}
\put(4,-56.5){\line(5,2){40}}
\put(19,-45){$6$}
\put(19,-60){$e_3$}

\put(42.4,-42.4){$\bullet$}
\put(48.4,-46.4){$v_4$}
\multiput(45,-41.4)(1,2.5){18}{$\cdot$}
\put(45,-20){$6$}
\put(60,-20){$e_4$}

\put(60,0){$\bullet$}
\put(67,0){$v_5$}
\put(62,4){\line(-2,5){17}}
\put(45,18){$6$}
\put(60,18){$e_5$}

\put(42.4,42.4){$\bullet$}
\put(48.4,46.4){$v_6$}
\put(43.4,46.4){\line(-5,2){40}}
\put(19,45){$1$}
\put(19,60){$e_6$}

\put(0,60){$\bullet$}
\put(0,68){$v_7$}
\multiput(2,60.5)(-2.5,-1){18}{$\cdot$}
\put(-19,45){$1$}
\put(-22,60){$e_7$}

\put(-42.4,42.4){$\bullet$}
\put(-52,49){$v_8$}
\multiput(-41.4,43.4)(-1,-2.5){19}{$\cdot$}
\put(-45,18){$6$}
\put(-63,18){$e_8$}

\end{picture}
\caption{}
\label{fig: labeling of the graph}
\end{figure}

The combinatorial data that give rise to a binomial
$f_\sigma^r=t^a-t^b$ is clarified by representing it in the graph $\G$, 
by putting a label $r$ or $\widehat{r}$ to each edge. 
Figure~\ref{fig: labeling of the graph} illustrates the map
$\rho_\sigma^6$ when  $q=8$, $r=6$, $s=8$ and 
$\sigma=\set{1,3,5,6}\sqcup\set{2,4,7,8}$. The labels of the edges 
correspond to the exponents of the variables in the corresponding
binomial of $I(X)$. Thus, \mbox{$f_\sigma^6 =
t_1^6t_3^6t_5^6t_6-t_2^6t_4^6t_7t_8^6$}.

\begin{lemma}\label{lemma: transferring an element from one part to the other}
Let $\sigma=A\sqcup B$ be a partition of $\set{1,\dots,s}$ and let $r\in \set{1,\dots,q-2}$. 
Suppose that $1\in A$ and that there exists $i\in A$ such that 
$i>2$ and $i-1\not \in A$. Let $\sigma'$ be the partition given by
$A'\sqcup B'$ where $A'=(A\setminus\set{i})\cup \set{i-1}$ and 
$B'=(B\setminus\set{i-1})\cup \set{i}$.
Then $f_\sigma^r\in I(X)$ if and only if $f_{\sigma'}^r \in I(X)$. 
\end{lemma}

\begin{proof}
Let $f_\sigma^r = t^a - t^b$. Using the assumption, we can write $t^a=t_i^{c}t^{a'}$
and $t^b=t_{i-1}^{c}t^{b'}$, where $c=a_i=b_{i-1}$ and $a',b'\in\NN^s$. Then:
\[\renewcommand{\arraystretch}{1.3}
\begin{array}{rcl}
(t_{i-1}t_i)^{\hat{c}}f^r_\sigma & = &t_{i-1}^{\hat{c}}t_i^{q-1}t^{a'} -t_{i-1}^{q-1}t_i^{\hat{c}}t^{b'}\\
& =& t_{i-1}^{\hat{c}}t_i^{q-1}t^{a'}-t_{i-1}^{\hat{c}}t_{i-1}^{q-1}t^{a'}+t_{i-1}^{\hat{c}}t_{i-1}^{q-1}t^{a'}
-t_{i-1}^{q-1}t_i^{\hat{c}}t^{b'} \\ 
& =& t_{i-1}^{\hat{c}}t^{a'}(t_i^{q-1}-t_{i-1}^{q-1})+(t_{i-1}^{\hat{c}}t^{a'}
-t_i^{\hat{c}}t^{b'})t_{i-1}^{q-1}. 
\end{array}
\]
Since $t_j$ is never zero on $X$ we get:
\[
f^r_\sigma \in I(X) \Leftrightarrow(t_{i-1}t_i)^{\hat{c}}f^r_\sigma
\in I(X) \Leftrightarrow 
(t_{i-1}^{\hat{c}}t^{a'} -t_i^{\hat{c}}t^{b'})t_{i-1}^{q-1} \in
I(X)\Leftrightarrow t_{i-1}^{\hat{c}}t^{a'} -t_i^{\hat{c}}t^{b'}\in I(X).
\]
Now let $a^\sharp,b^\sharp \in \NN^s$ be such that $t^{a^\sharp}=t_{i-1}^{\hat{c}}t^{a'}$
and $t^{b^\sharp}=t_i^{\hat{c}}t^{b'}$. Then, 
$\sigma'=\supp(a^\sharp)\sqcup \supp(b^\sharp)$ 
is the partition of $\set{1,\dots,s}$ obtained from 
interchanging $i-1$ and $i$ in $A\sqcup B$.
Applying Lemma~\ref{lemma: technical lemma} twice, we deduce that  
$f^r_{\sigma'}=t_{i-1}^{\hat{c}}t^{a'} -t_i^{\hat{c}}t^{b'}$.
%It is clear that $f_\sigma^r$ is homogeneous if and only if $f_{\sigma'}^r$ is.
\end{proof}

\begin{lemma}\label{lemma: condition on homogeneity}
Let $\sigma=A\sqcup B$ be a partition of $\set{1,\dots,s}$  with $1\in A$
and let $r\in\set{1,\dots,q-2}$. If $f^r_\sigma\in I(X)$ then $\ls{A}=\ls{B}$.
\end{lemma}

\begin{proof}
Let $\ell=\ls{A}$.
\mbox{Using} sufficiently many times 
Lemma~\ref{lemma: transferring an element from one part to the other}, we may assume
that $\sigma$ is the partition $\set{1,\dots,\ell}\sqcup
\set{\ell+1,\dots,s}$. 
\mbox{Accordingly,} 
\[
f_\sigma^r=t_1^r t_2^{\hat{r}}\cdots t_\ell^{r'}-t_{\ell+1}^{r'}\cdots t_{s-1}^{\hat{r}}t_s^r, 
\]
where $r'\in \set{r,\hat{r}}$. 
Now $\deg(t^{{r}}_{j}\cdots )$, for a monomial consisting 
of a product of variables with consecutive exponents alternating in $\set{r,\hat{r}}$, is a 
strictly increasing 
function with res\-pect to the number of variables involved. 
Since $f^r_\sigma$ is homogeneous (by Lemma~\ref{h-lemma}) we deduce that 
$\ls{A}=\ell=s-\ell=\ls{B}$.  
\end{proof}

We come to one of the main results of this section, a combinatorial
description of a generating set for a vanishing ideals over 
an even cycle.

\begin{theorem}\label{theorem: the generators of I(X)} Let $I(X)$ be
the vanishing ideal of the algebraic toric set $X$ associated to an
even cycle 
$\G=\C_{2k}$. Then, $I(X)$ is generated by the binomials $t_i^{q-1}-t_j^{q-1}$, 
$1\leq i,j\leq s=2k$, and the binomials $f_\sigma^r$ obtained from
 all $r\in \set{1,\dots,q-2}$ and all partitions $\sigma=A\sqcup B$ 
of $\set{1,\dots,s}$ with $\ls{A}=\ls{B}$.
\end{theorem}

\begin{proof}
By Theorem~\ref{lemma: first reduction} 
and Corollary~\ref{corollary: supports add up to all}, we know that 
$I(X)$ is generated by the binomials of the form
$t_i^{q-1}-t_j^{q-1}$, $1\leq i,j\leq s=2k$, and the homogeneous binomials $f=t^a-t^b$ with 
$a=(a_1,\dots,a_s)\in \NN^s$ and $b=(b_1,\dots,b_s)\in \NN^s$ such that 
$\supp(a)\sqcup \supp(b)=\set{1,\dots,s}$ and $0\leq a_i,b_j \leq q-2$. 
Let $f$ be a binomial of
the latter type. We may assume that $1\in{\rm supp}(a)$, for we can always replace
$f$ by $-f$ in a generating set of $I(X)$.
Set $\sigma = \supp(a)\sqcup \supp(b)$ and let $r=a_1$. Let us show that 
$f=f_\sigma^r$, i.e., let us show that $a_{i}=\rho_\sigma^r(i)$, for
every $i\in \supp(a)\setminus \set{1}$ 
and $b_j=\rho_\sigma^r(j)$ for every $j\in\supp(b)$. Let $i\in \supp(a)\setminus \set{1}$ and
let $u\in K^*$ be a generator of the multiplicative group of $K$.
Consider $\xx\in (K^*)^n$ given by setting $x_i=u$ and $x_j=1$ for all $j\not = i$. 
Then, $f(\xx^{\nu_1},\ldots,\xx^{\nu_s})=0$ implies that
$u^{a_{i-1}}u^{a_{i}}=1$, if $i-1\in \supp(a)$ or  
$u^{a_i}=u^{b_{i-1}}$ if $i-1\in \supp(b)$. We get, in the first
case, $a_i=q-1-a_{i-1}=\rho_\sigma^r(i)$,  
and, in the second case, $a_i=b_{i-1}=\rho_\sigma^r(i)$. Similarly, 
if $j\in \supp(b)$, then $b_j=\rho_\sigma^r(j)$. Since $f^r_\sigma \in I(X)$, by 
Lemma~\ref{lemma: condition on homogeneity}, $\ls{A}=\ls{B}$.

\noindent
To complete the proof let $\sigma=A\sqcup B$ be a partition 
of $\set{1,\dots,s}$ with $\ls{A}=\ls{B}$,  \mbox{$r\in \set{1,\dots,q-2}$}
and let us show that $f^r_\sigma \in I(X)$. 
By Lemma~\ref{lemma: transferring an element from one part to the other}, we may assume
that $\sigma$ is the partition $\set{1,\dots,k}\sqcup
\set{k+1,\dots,s}$ and $f_\sigma^r=t_1^r t_2^{\hat{r}}\cdots t_k^{r'}-t_{k+1}^{r'}\cdots t_{s-1}^{\hat{r}}t_s^r$,
where $r'\in \set{r,\hat{r}}$. Now, let $\xx \in (K^*)^n$. 
Then
$f_\sigma^r(\xx^{\nu_1},\ldots,\xx^{\nu_s})=x_1^rx_{k+1}^{r'}-x_{k+1}^{r'}
x_1^r =  0$, {\it{i.e.}}, $f_\sigma^r\in I(X)$. 
\end{proof}

Te following conjecture has been verified in a number of examples
using {\it{Macaulay}}$2$ \cite{mac2}. 

\begin{conjecture}\label{remark: on the degree of generators for even
cycles}\rm 
Let $X$ be the algebraic toric set associated to an even cycle
$\G=\C_{2k}$ and let $\lambda$ be the partition
$\set{1,3,\dots,2k-1}\sqcup \set{2,4,\dots,2k}$. 
If $k\geq 2$, 
then the set of binomials 
\[\renewcommand{\arraystretch}{1.3}
\begin{array}{c}
\mathcal{B}= \{\{f^r_\sigma\colon \sigma=A\sqcup B\mbox{ is a partition 
of } \{1,\ldots,s\}\mbox{ with }|A|=|B|, 1\in A\mbox{ and }1\leq r\leq q-2 \}\\ 
\cup\{t_i^{q-1}-t_s^{q-1}\colon\, 1\leq i \leq s-1\}\}\setminus\{f^r_\lambda\colon\,2\leq r \leq q-2\} 
\end{array}
\]
is a minimal set of generators and a Gr\"{o}bner basis of $I(X)$ with 
respect to the reverse lexicographic order.
\end{conjecture}

\begin{remark}\label{rmk: About the generating set of I(X)}
By Theorem~\ref{theorem: the generators of I(X)} and 
since, for each $2\leq r \leq q-2$, there exists $g_r\in S$ such that $f^r_\lambda = g_rf^1_{\lambda}$,
we get that $\mathcal{B}$ is a generating set for $I(X)$.

\end{remark}

\begin{corollary}\label{remark: on the degrees of the generators of I(X)}
Let $\G=\C_{2k}$ be an even cycle. 

\begin{itemize}
\item[(a)] If $f=t^a-t^b$ is an element of $\mathcal{B}$, then $\deg(f)$ is at most $(q-2)(k-1)+1$.
% \item[(b)] There is a partition $\sigma$ such that 
% $\deg(f_\sigma^{q-2})=(q-2)(k-1)+1$ and $f_\sigma^{q-2}$ is part of 
% a minimal generating set of $I(X)$ consisting of binomials. 

\item[(b)] Any subset of $\mathcal{B}$ that is also generating set of 
$I(X)$ contains an element of the form $f^{q-2}_\sigma$,  
with $\deg(f_\sigma^{q-2})=(q-2)(k-1)+1$.
\end{itemize}
\end{corollary}

\begin{proof} We set $M=(q-2)(k-1)+1$. (a) If $f=t_i^{q-1}-t_j^{q-1}$ for
some $i,j$, then $\deg(f)\leq M$ because $k\geq 2$. Assume that $f$ 
is not of this form. Then, $f=f_\sigma^r$ for some $1\leq
r\leq q-2$, where $\sigma$ is the partition $\sigma=A\sqcup B$ and
$A$, $B$ are the supports of $a$, $b$ respectively, $\ls{A}=\ls{B}=k$ and
$1\in A$. Let $i$ be the
cardinality of the set $\{j\in A\colon\,f_\sigma^r(j)=r\}$. 
Then, $\deg(f)=ir+(k-i)\widehat{r}$. If $r=\widehat{r}$, then
$r=(q-1)/2$ and $\deg(f)=k(q-1)/2\leq M$. We may now assume
$r\neq\widehat{r}$. If $i=k$, then 
$$f^r_\sigma = f^1_\lambda = t_1t_3\cdots t_{2k-1}-t_2t_4\cdots t_{2k}.$$
Hence, $\deg(f_\sigma^r)=k\leq M$. To complete the proof we
may now assume that $1\leq i\leq k-1$. In this case, we have
\begin{eqnarray*}
\deg(f_\sigma^r)&=&ir+(k-i)\widehat{r}=i(r-\widehat{r})+k\widehat{r}
\\
&\leq&(k-1)(r-\widehat{r})+k\widehat{r}=(k-1)r+\widehat{r}=r(k-2)+(q-1)\\
&\leq&(q-2)(k-2)+(q-1)=(k-1)(q-2)+1.
\end{eqnarray*}
Thus, $\deg(f_\sigma^r)\leq M$, as required. To prove (b) consider
$\mathcal{B}'\subset \mathcal{B}$ be a generating set of $I(X)$.
Let \mbox{$\sigma=\{1,3,\ldots,2k-3,2k\}\sqcup\{2,4,\ldots,2(k-1),2k-1\}$}, then
$$
f_\sigma^{q-2}=t_1^{q-2}t_3^{q-2}\cdots t_{2k-3}^{q-2}t_{2k}-t_2^{q-2}t_4^{q-2}\cdots
t_{2(k-1)}^{q-2}t_{2k-1}
$$
is in $\mathcal{B}$ and has degree $M$. We will show that $f_\sigma^{q-2} \in \; \mathcal{B}'$. 
Since $\mathcal{B}'$ is a generating set of $I(X)$, $f_\sigma^{q-2}$ is a
linear combination, with coefficients in $S$, 
of binomials in $\mathcal{B}'$. These are binomials of the form $f_\rho^m$, $1\leq m\leq
q-2$, $\deg(f_\rho^{m}) \leq M$, and of the form $t_i^{q-1}-t_{2k}^{q-1}$ for
some $1\leq i\leq 2k-1$. It is seen that there is a
binomial in $\mathcal{B}'$, $f_\rho^{m}=t^{a'}-t^{b'}$,  such that  
$$t_1^{q-2}t_3^{q-2}\cdots t_{2k-3}^{q-2}t_{2k}=
t^ct^{a'}
$$
for some monomial $t^c$. Since $\supp (a')\subset \set{1,3,\dots,2k-3,2k}$ and $t^{a'}-t^{b'}$ cannot
be of the form $t_1^{q-1}-t_{2k}^{q-1}$, because of its degree, we deduce that $\sigma=\rho$. From the equality 
$$t_1^{q-2}t_3^{q-2}\cdots t_{2k-3}^{q-2}t_{2k}=t^ct^{a'}=
t^c(t_1^{m}t_3^{m}\cdots t_{2k-3}^{m}t_{2k}^{\widehat{m}})
$$
we conclude that $\widehat{m}=1$, that is, $m=q-2$. Thus,
$f_\sigma^{q-2}=f_\rho^m$.
\end{proof}

Consider the general case when $\G$ is any graph. Suppose that 
$\G$ contains a subgraph \mbox{$\mathcal{H}\cong \C_{2k}$}, 
isomorphic to an even order cycle. Assume without loss of generality
that $t_1,\dots,t_{2k}$ are the variables of $S$ corresponding to the edges of
$\mathcal{H}$. Then, given $r\in \set{1,\dots,q-2}$ and a partition $\sigma=A\sqcup B$ of 
$\set{1,\dots, 2k}$ with $\ls{A}=\ls{B}=k$ and $1 \in A$, the homogeneous binomial
$f^r_\sigma\in K[t_1,\dots,t_{2k}]\subset S$ clearly vanishes on the
algebraic toric set associated to 
$\G$. One could conjecture that together with the binomials
$t_i^{q-1}-t_j^{q-1}$, for $1\leq i,j\leq s$, 
the binomials obtained in this way, going through all the even cycles
of $\G$, would form a generating 
set of $I(X)$. This is not true, even for bipartite graphs, as is
shown by 
Example~\ref{example: two offending graphs}. This conjecture 
is true if we restrict to bipartite graphs the cycles of which are
vertex disjoint; as we show in  
Theorem~\ref{theorem: I(X) for almost general graph}.

Suppose $\G$ is a bipartite graph the cycles of which have disjoint vertex sets. 
Let $\H_1,\dots,\H_m$ be the subgraphs of $\G$ isomorphic to some even order cycle, i.e., such
that $\H_i\cong \C_{2k_i}$. Let
$t_{\epsilon^i_1},\dots,t_{\epsilon^i_{2k_i}}\in S$ be the variables  
associated to the edges, $e^i_1,\dots,e^i_{2k_i}$ of $\H_i$. Accordingly, set 
\[
S_i=K\bigl[t_{\epsilon^i_1},\dots,t_{\epsilon^i_{2k_i}}\bigr]\subset S.
\]
 Finally, denote
by $I_i(X)$ the intersection $I(X)\cap S_i$. Then, $I_i(X)\subset S_i$ is equal to $I(X_i)$, 
the vanishing ideal of the algebraic toric set $X_i$ associated to $\H_i$.

\begin{theorem}\label{theorem: I(X) for almost general graph}
Let $\G$ be a connected bipartite graph, whose {\rm(}even{\rm)} cycles $\H_1,\dots,\H_m$ have
disjoint vertex sets. Let $X$ be the algebraic toric set associated to $\G$. 
Then $I(X)$ is generated by the union of 
the set $\{t_i^{q-1}-t_j^{q-1}:1\leq i,j\leq s\}$ with the set
$I_1(X)\cup\cdots\cup I_m(X)$.
\end{theorem}

\begin{proof}
By Theorem~\ref{lemma: first reduction}, it suffices to show that if 
$f=t^a-t^b\in I(X)$, with $a=(a_1,\dots,a_s)\in \NN^s$,
$b=(b_1,\dots,b_s)\in \NN^s$,
such that $\supp(a)\cap \supp(b)=\emptyset$ and $1\leq a_i,b_j\leq
q-2$, then $f$ belongs to the ideal generated by
\[
\mathcal{J}=\{t_i^{q-1}-t_j^{q-1}:1\leq i,j\leq s\} \cup  I_1(X)\cup\cdots\cup I_m(X).
\]
Recall that $f$ is homogeneous by Lemma~\ref{h-lemma}. 
By Proposition~\ref{proposition: nonoccuring variables}, we know that $\supp(a)\cup\supp(b)$
is contained in the union of the sets of indices of the variables corresponding to edges 
of the cycles of $\G$. In other words, if $e_i$ is an edge not in any edge set of $\H_1,\dots,\H_m$
then $i\not \in \supp(a)\cup\supp(b)$.  
As above, denote 
by $t_{\epsilon^i_1},\dots,t_{\epsilon^i_{2k_i}}$ the variables associated to $\H_i$. 
We proceed by induction on 
\[
 \mu_f=\set{i\in \set{1,\dots,m} : (\supp(a)\cup \supp(b))\cap
 \smset{\epsilon^i_1,\dots,\epsilon^i_{2k_i}} 
\not = \emptyset }.
\]

\noindent
Let $i\in\set{1,\dots,m}$ be such that 
$(\supp(a)\cup\supp(b))\cap\smset{\epsilon^i_1,\dots,\epsilon^i_{2k_i}}\not = \emptyset$.
Consider \mbox{$a^\sharp,a^\flat,b^\sharp,b^\flat\in \NN^s$}  such that  
$\supp(a^\sharp)\cup \supp(b^\sharp)\subset \smset{\epsilon^i_1,\dots,\epsilon^i_{2k_i}}$,
$(\supp(a^\flat)\cup \supp(b^\flat))\cap
\smset{\epsilon^i_1,\dots,\epsilon^i_{2k_i}}=\emptyset$, 
\[t^a=t^{a^\sharp} t^{a^\flat}\quad \text{and}\quad t^b=t^{b^\sharp} t^{b^\flat}.
\]
By Corollary~\ref{corollary: supports add up to all},
\mbox{$\supp(a^\sharp)\cup \supp(b^\sharp)= \smset{\epsilon^i_1,\dots,\epsilon^i_{2k_i}}$}.
Since we are assuming $\H_1,\dots,\H_m$ have disjoint vertex sets,
setting $t_\ell=1$ for all
$\ell\not\in \set{\epsilon^i_1,\dots,\epsilon^i_{2k_i}}$ is
equivalent to setting in \mbox{$\xx\in (K^*)^n$},  
$x_\ell=1$ for all 
$\ell\not \in V_{\H_i}$. Hence, making these substitutions and
running the argument of the proof of  
Theorem~\ref{theorem: I(X) for almost general graph}, we see that $t^{a^\sharp}-t^{b^\sharp}=
f_\sigma^r$, where \mbox{$r=(a^{\sharp})_{\epsilon_1^i}\in
\set{1,\dots,q-2}$}, (assu\-ming that $\epsilon^i_1\in
\supp(a^\sharp)$), and where $\sigma$ is the partition   
\mbox{$\supp(a^\sharp)\sqcup
\supp(b^\sharp)=\smset{\epsilon^i_1,\dots,\epsilon^i_{2k_i}}$.} 
\smallskip 

\noindent
Suppose that $\mu_f=1$. Then $a^\flat=b^\flat=0\in\NN^s$, 
$f^r_\sigma$ is homogeneous and we are done. 
\medskip

\noindent
Suppose that every binomial $g=t^a-t^b\in I(X)$ with 
$\mu_g\leq m'< m$ is in the ideal generated by $\mathcal{J}$.  
Let $f=t^a-t^b\in I(X)$ be a binomial with $\mu_f=m'+1$. 
Let $i\in\set{1,\dots,m}$ be such that 
$(\supp(a)\cup\supp(b))\cap\smset{\epsilon^i_1,\dots,\epsilon^i_{2k_i}}\not = \emptyset$.
Consider, as above, \mbox{$a^\sharp,a^\flat,b^\sharp,b^\flat\in \NN^s$} such that 
$t^a=t^{a^\sharp} t^{a^\flat}$ and $t^b=t^{b^\sharp} t^{b^\flat}$. 
Repeating the previous argument we deduce that $t^{a^\sharp}-t^{b^\sharp}=f^r_\sigma$ where, 
$r=(a^\sharp)_{\epsilon^i_1}$ and $\sigma=\supp(a^\sharp)\sqcup \supp(b^\sharp)$.
However, notice that in this case $f^r_\sigma$ is not necessarily homogeneous.
Assume that \mbox{$\smls{\supp(a^\sharp)}\geq \smls{\supp(b^\sharp)}$}.
Let $\delta\in \NN^s$ be such that $\epsilon^i_1\not \in \supp(\delta)\subset \supp(a^\sharp)$,
$\delta_\ell=a^\sharp_\ell$ for all $\ell\in \supp(\delta)$ and $\ls{\supp(a^\sharp -\delta)}=k_i$ (where 
$2k_i$ is the order of $\mathcal{H}_i$). Set $h=\ls{\supp(\delta)}$, $a'=a^\sharp - \delta$ and let 
$b'\in \NN^s$ be obtained by applying $h$ times Lemma~\ref{lemma:
transferring an element from one part to the other} 
to $\sigma=\supp(a^\sharp)\sqcup \supp(b^\sharp)$. Then $b'=b^\sharp + \hat\delta$, where
$\hat\delta$ has the same support as $\delta$ and $(\hat\delta)_\ell=q-1-\delta_\ell$, for every 
$\ell\in \supp(\hat\delta)$. Set $\sigma' = \supp(a')\sqcup \supp(b')$. 
Then $f_{\sigma'}^r=t^{a'}-t^{b'}$ is 
homogeneous and belongs to $I_i(X)$. Moreover,
\begin{equation}
\begin{array}{c}
f = t^a-t^b = t^{a'}t^\delta t^{a^\flat} - t^{b^\sharp}t^{b^{\flat}}=
t^{a'}t^\delta t^{a^\flat}-t^{b'}t^\delta t^{a^\flat}+t^{b'}t^\delta
t^{a^\flat} - t^{b^\sharp}t^{b^{\flat}}\\ 
=f^r_{\sigma'} t^\delta t^{a^\flat} +
t^{b^\sharp}(t^{\widehat{\delta}} t^\delta t^{a^\flat} -t^{b^\flat}). 
\end{array}
\end{equation}
Now $(\hat{\delta})_\ell+\delta_\ell=q-1$, for all $\ell\in \supp(\delta)$ and 
since $f$ is homogeneous, $h=\smls{\supp(\delta)}>\smls{\supp(b^\flat)}$. Choose 
$\ell_1,\dots,\ell_h\in \supp(b^\flat)$, $h$ distinct indices. Let $\gamma \in \NN^s$ to be such that
$\supp(\gamma)=\set{\ell_1,\dots,\ell_h}$ and $(\gamma)_{\ell_j}=q-1$, for $j=1,\dots,h$.
Then $t^{\delta}t^{\widehat{\delta}}-t^\gamma$ is in the ideal of $S$ generated by $\mathcal{J}$,
since it is in the ideal of the torus. We have
\begin{equation}
f= f^r_{\sigma'} t^\delta t^{a^\flat} +
t^{b^\sharp}(t^{\delta}t^{\widehat{\delta}} t^{a^\flat} -t^{b^\flat}) 
=f^r_{\sigma'} t^\delta t^{a^\flat} + 
t^{b^\sharp}t^{a^{\flat}}(t^{\delta}t^{\widehat{\delta}}-t^\gamma)
+t^{b^\sharp}( t^\gamma t^{a^\flat} -t^{b^\flat}).
\end{equation}
Let $\gamma^\sharp\in \NN^s$ be such that $\supp(\gamma^\sharp)=\set{\ell_1,\dots,\ell_h}$ and
$(\gamma^\sharp)_{\ell_j}=(b^\flat)_{\ell_j}$, for $j=1,\dots,h$ and
set $\gamma^\flat = \gamma - \gamma^\sharp$ 
and $b^\natural = b^\flat - \gamma^\sharp$. Then,
\begin{equation}
f= f^r_{\sigma'} t^\delta t^{a^\flat} + 
t^{b^\sharp}t^{a^{\flat}}(t^{\delta^*}-t^\gamma)
+t^{b^\sharp}t^{\gamma^\sharp}( t^{\gamma^\flat} t^{a^\flat} -t^{b^\natural}),
\end{equation}
where $g= t^{\gamma^\flat} t^{a^\flat} -t^{b^\natural}$ is a homogeneous binomial
with $\mu_g\leq m'$. Hence, by induction, $g$, and therefore $f$, are
in the ideal generated by $\mathcal{J}$. 
\end{proof}

In Example~\ref{example: two offending graphs}, we show 
that Theorem~\ref{theorem: I(X) for almost general graph} does not hold for general connected 
bipartite graphs.

\begin{figure}[h]
\begin{picture}(200,130)(-90,-140)

\put(-170.5,-80){$\circ$}
\put(-177.5,-80){$1$}
\put(-166,-76){\line(1,1){40}}
\put(-157.5,-55){$e_1$}
\put(-145,-65){$\scriptstyle 1$}

\put(-130,-40){$\bullet$}
\put(-130,-32.5){$2$}
\multiput(-129,-40)(2.2,-2.2){18}{$\cdot$}
\put(-108,-55){$e_2$}
\put(-115,-65){$\scriptstyle 1$}

\put(-130,-120){$\bullet$}
\put(-130,-130){$4$}
\put(-127.5,-117.75){\line(-1,1){39}}
\multiput(-129,-120)(2.2,2.2){18}{$\cdot$}
\put(-107,-103){$e_3$}
\put(-115,-95){$\scriptstyle 1$}
\put(-157.5,-103){$e_4$}
\put(-145,-95){$\scriptstyle 1$}

\put(-90.5,-80){$\bullet$}
\put(-90.5,-70){$3$}
\multiput(-89,-80)(2.2,2.2){18}{$\cdot$}
\multiput(-89.25,-80)(2.2,-2.2){18}{$\cdot$}
\put(-78,-55){$e_5$}
\put(-65,-65){$\scriptstyle 1$}

\put(-50,-40){$\bullet$}
\put(-50,-32.5){$5$}
\put(-47,-38){\line(1,-1){40}}
\put(-26,-55){$e_6$}
\put(-33,-65){$\scriptstyle 1$}

\put(-50,-120){$\bullet$}
\put(-50,-130){$6$}
\put(-48,-118){\line(1,1){40}}
\put(-78,-103){$e_8$}
\put(-65,-95){$\scriptstyle 1$}

\put(-10.5,-80){$\bullet$}
\put(-2.5,-80){$1$}
\put(-26,-103){$e_7$}
\put(-33,-95){$\scriptstyle 1$}

%%%
%% SECOND GRAPH
%%%

\put(100,-20){$\bullet$}
\put(100,-12.5){$2$}
\multiput(101,-20)(-2.2,-2.2){28}{$\cdot$}
\put(62.5,-42.5){$e_1$}
\put(74,-54){$\scriptstyle 1$}
\multiput(101,-20)(2.2,-2.2){28}{$\cdot$}
\put(135,-42.5){$e_2$}
\put(127,-54){$\scriptstyle 1$}

\put(40,-80){$\bullet$}
\put(32.5,-80){$1$}
\put(100,-80){$\bullet$}
\put(107.5,-80){$5$}
\multiput(101.5,-80)(0,3.2){20}{$\cdot$}
\put(107.5,-100.5){$e_5$}
\put(92,-101){$\scriptstyle 2$}
\put(103,-78){\line(0,-1){60}}
\put(107.5,-59.5){$e_6$}
\put(92,-60){$\scriptstyle 2$}
\put(160,-80){$\bullet$}
\put(167.5,-80){$3$}

\put(100,-140){$\bullet$}
\put(100,-150){$4$}
\put(102,-137){\line(-1,1){60}}
\put(74,-107){$\scriptstyle 1$}
\put(62.5,-115){$e_4$}
\put(102,-138){\line(1,1){60}}
\put(127,-107){$\scriptstyle 1$}
\put(135,-115){$e_3$}

\end{picture}
\caption{}
\label{fig: two offending graphs}
\end{figure}

\begin{example}\label{example: two offending graphs}
Let $\G_1$ and $\G_2$ be the two graphs in Figure~\ref{fig: two offending graphs} 
(from left to right) and assume that $q=5$. Notice that we are
identifying the two vertices, labeled by $1$, in the representation
of $\G_1$. Thus, $\G_1$ is a bipartite graph with six vertices and
eight edges.    
Denote by $X_1$ and $X_2$, respectively, the corresponding algebraic toric sets.
Then, using \emph{Macaulay}$2$ \cite{mac2}, 
we found that the binomial $t_1t_4t_6t_7-t_2t_3t_5t_8$ is in a minimal 
generating set of $I(X_1)$. In this case, the argument of the proof of 
Theorem~\ref{theorem: I(X) for almost general graph} does not work, to the extent that if we set
$t_1,t_2,t_3,t_4$ equal to $1$, the resulting binomial, $t_6t_7-t_5t_8$, albeit homogeneous, is not
of the type $f^r_\sigma$ for any partition $\sigma$ of $\set{5,6,7,8}$. The same can be said for the
binomial resulting from substituting to $1$ the variables $t_5,t_6,t_7,t_8$. As to the vanishing
ideal of $X_2$, we found that there exists a minimal generating set containing 
$t_1t_2t_5^2-t_3t_4t_5^2$, which, when restricted to any of the $3$ cycles in $\G_2$ is not of the 
type $f^r_\sigma$ for any partition of the corresponding index set.

\end{example}

\section{The regularity of $R/I(X)$}\label{sec: regularity}

In this section we address the question of computing the regularity of $S/I(X)$ for an algebraic
toric set $X$ parameterized by a bipartite graph. 
Theorem~\ref{theorem: bound on regularity for a general graph}
gives an upper bound for the regularity of $S/I(X)$ for a general bipartite
graph. If $X$ is the algebraic toric 
set parameterized by an even cycle of length $2k$, 
by Proposition~\ref{castelnuovo-vs-bigdegree} and 
Corollary~\ref{remark: on the degrees of the generators of I(X)}, we
get 
$${\rm bigdeg}\, I(X)-1=(q-2)(k-1)\leq {\rm reg}\, S/I(X).$$
This inequality is already known in the literature, 
see \cite[Corollary 3.1]{codes over cycles} and \cite[Corollary
2.19]{deg-and-reg}. We will show that the regularity of $S/I(X)$ 
is in fact equal to $(q-2)(k-1)$, and generalize this result by giving
a formula for the regularity of any 
connected bipartite graph whose cycles have disjoint
vertex sets. In the proof of 
Theorem~\ref{theorem: regularity of even cycles}, we show the 
inequality above as an easy consequence of the description of the generators of the
ideal $I(X)$.

\begin{lemma}\label{lemma: symmetries on the generating set}
Let $1\leq i\leq s-2$. Consider the $K$-automorphism $\sigma_{i}\colon S \rt S$ defined
by exchanging $t_i$ with $t_{i+2}$ and leaving all other variables fixed.
Then, $\sigma_{i}$ permutes the elements of the set of all 
$f^r_\sigma\in I(X)$, for $r\in \set{1,\dots,q-2}$ and $\sigma=A\sqcup B$ a partition of
$\set{1,\dots,s}$ with $\ls{A}=\ls{B}$.
\end{lemma}

\begin{proof}
Let $f^r_\sigma$ be a binomial associated to $r\in\set{1,\dots,q-2}$ and $\sigma=A\sqcup B$ a partition 
of $\set{1,\dots,s}$. Thus, 
$f^r_\sigma = t^a-t^b$ where  $A=\supp(a)$, $B=\supp(b)$,
$a_\ell=\rho_\sigma^r(\ell)$, for all $\ell\in \supp(a)$ 
and $b_\ell=\rho^r_\sigma(\ell)$, for all $\ell\in \supp(b)$. As 
$\rho_\sigma^r(\ell)=\rho_\sigma^r(\ell+2)$ if
and only if $\ell$ and $\ell+2$ are in the same part of the
partition, if $i$ and $i+2$ are in the same part of the partition then 
$\sigma_i(f^r_\sigma)=f^r_\sigma$. Suppose that $i$ and $i+2$ are in different parts of the 
partition and therefore that $\rho_\sigma^r(i+2)=\widehat{\rho_\sigma^r(i)}$. Without loss in generality
we may write $f^r_\sigma=t_i^{a_i}t^{a'}-t_{i+1}^{\widehat{a_i}}t^{b'}$, where 
$\supp(a')=\supp(a)\cup\set{i}$ and $\supp(b')=\supp(b)\cup\set{i+2}$.
In this situation, 
we apply Lemma~\ref{lemma: transferring an element from one part to
the other} twice, transferring
$i$ to the part it does not belong to, and proceeding similarly with $i+2$. Let 
$\sigma'$ be the partition of $\set{1,\dots,s}$ obtained in this way and consider the resulting binomial 
$f^r_{\sigma'}$. By Lemma~\ref{lemma: transferring an element from one part to the other} we see that 
$f^r_{\sigma'}=t_{i+2}^{a_i}t^{a'}-t_{i}^{\widehat{a_i}}t^{b'}=\sigma_i(f^r_\sigma)$.
\end{proof}

\begin{theorem}\label{theorem: regularity of even cycles}
Let $X$ be the algebraic toric set associated to an even order cycle $\G=\C_{2k}$.
Then $\reg S/I(X) =(q-2)(k-1)$.
\end{theorem}

\begin{proof} Recall that $k\geq 2$. Denote by $R$ the graded ring 
$S/I(X)$.
Consider $t_1\in S$. Since $t_1$ is regular on $R$, 
we have the following exact sequence of
graded $S$-modules:
\begin{equation}\label{eq: SES of graded modules}
0\longrightarrow R[-1]\stackrel{t_1}{\longrightarrow}R\longrightarrow
R/(t_1)\longrightarrow 0,
\end{equation}
where $R[-1]$ is the graded $S$-module obtained by a shift
in the graduation, {\it{i.e.}}, $R[-1]_i=R_{i-1}$. Recall
that $H_X(d)$ is, by definition, $\dim_K (S/I(X))_d$, and since $S/I(X)$ is a $1$-dimensional
ring, the regularity of $S/I(X)$ is the least integer $l$ for which $H_X(d)$ is equal 
to some constant (indeed equal to $\ls{X}$) for all $d\geq l$. Now,
from (\ref{eq: SES of graded modules}) we get $H_X(d)-H_X(d-1)=\dim_K
(R/(t_1))_d$.   
Hence $\reg S/I(X)= \reg R/(t_1)-1$. 
For $d \geq 0$, we define
$$
h_d:= \dim_K
(R/(t_1))_d = H_X(d)-H_X(d-1) .
$$
We start by showing that $\reg S/I(X)\leq (q-2)(k-1)$.
If we show that $h_d=0$, for $d\geq (q-2)(k-1)+1$,
then $H_X(d-1)=H_X(d)$, for $d-1\geq (q-2)(k-1)$,
and our result follows.
Set $S'=K[t_2,\dots,t_s]$.  There is a surjection
of graded $S'$-modules
\[
\varphi\colon S'\longrightarrow
S/(I(X),t_1)\cong R/(t_1)
\]
defined by $\varphi (f)=f+(I(X),t_1)$, for every $f\in S'$. 
Set $I'(X)=\mbox{Ker}(\varphi)$, so that 
\[
S'/I'(X) \cong {S}/(I(X),t_1).
\]
Then, $I'(X)$ 
is a monomial ideal generated by the monomials obtained by setting $t_1=0$ in the generators of $I(X)$;
in particular it is generated by $t_j^{q-1}$, for $2\leq j\leq s$ and by the monomials
$t^b$ in some $f^r_\sigma=t^a-t^b$, for $r\in \set{1,\dots,q-2}$ and $\sigma$ a partition
of $\set{1,\dots,s}$ into $2$ parts of equal cardinality.
To show that $h_d=0$, for $d\geq (q-2)(k-1)+1$, it is enough to show that every monomial $M$ in
$S'$ of degree $\geq (q-2)(k-1)+1$ belongs to $I'(X)$. 
Since $t_j^{q-1}\in I'(X)$ for all $2\leq j\leq s$, we may
assume that there is no $j$ for which $t_j^{q-1}$ divides the monomial $M$
in question. Let us write it in the following way:
\[
M=t_2^{b_1}t_4^{b_2}\cdots t_{2k}^{b_k}\,t_3^{c_1}t_5^{c_2}\cdots t_{2k-1}^{c_{k-1}},
\]
with $0\leq b_i,c_j\leq q-2$. We want to show that there exists $f^r_\sigma=t^a-t^b\in I(X)$
such that $t^b$ divides $M$. By Lemma~\ref{lemma: symmetries on the generating set},
if $t^b$ divides $M$ and there exists $r,\sigma$ such that $f^r_\sigma=t^a-t^b$,
then, for all $i\in \set{2,\dots,s-2}$, $\sigma_i(t^b)$ divides 
$\sigma_i(M)$ and there exists $\sigma'$ such that 
$f^r_{\sigma'}=t^{a'}-\sigma_i(t^b)$. 
Hence, we may assume that $c_1\leq c_2\leq \cdots \leq c_{k-1}$
and that $b_1\geq b_2 \geq \cdots \geq b_{k}$. There are two cases.
If $b_{k}>0$, then $M$ is divisible by $t_2t_4\cdots t_{2k}$, which belongs to $I'(X)$, 
since for $\sigma=\set{1,3,\dots,2k-1}\sqcup \set{2,4,\dots,2k}$, 
we have $f^1_\sigma=t_1t_3\cdots t_{2k-1}-t_2 t_4\cdots t_{2k}$.
The second case is for $b_k=0$. In this case, from 
\[
\deg M=\sum_{i=1}^{k-1}(b_i+c_i)\geq (q-2)(k-1)+1 
\]
we deduce that there exists $j \in \{1, \ldots, k-1\}$ such that
$b_j+c_j \geq q-1$. Since $c_{j}\leq q-2$ we get $b_j\geq 1$. 
Set $r=b_j$. Notice that then $c_j\geq q-1-b_j=q-1-r=\widehat{r}$. Consider the set 
given by \mbox{$B=\set{2,4,\dots,2j,2j+1,2j+3,\dots,2k-1}$}
and let $\sigma=A \sqcup B$ be the partition of $\set{1,\dots,s}$ that it determines. Then:
\[
f^r_\sigma= 
(t_1 t_3\cdots t_{2j-1})^r (t_{2j+2}\cdots t_{2k-2} t_{2k})^{\widehat{r}} -
(t_2 t_4\cdots t_{2j})^r (t_{2j+1} t_{2j+3}\cdots t_{2k-1})^{\widehat{r}}\in I(X). 
\]
Accordingly, $(t_2 t_4\cdots t_{2j})^r (t_{2j+1} t_{2j+3}\cdots t_{2k-1})^{\widehat{r}}\in I'(X)$.
Since $b_l\geq b_j=r$, for all $1\leq l\leq j$, we deduce that  $t_{2l}^r$ divides $M$,
for all $1\leq l\leq j$. Since $\widehat{r}\leq c_{j}\leq c_l$, for all $j\leq l \leq k-1$, 
we deduce that $t_{2l+1}^{\widehat{r}}$ divides $M$, for all $j\leq l \leq k-1$. In conclusion,
$(t_2 t_4\cdots t_{2j})^r (t_{2j+1} t_{2j+3}\cdots t_{2k-1})^{\widehat{r}}$ divides $M$ and hence
$M\in I'(X)$.

\medskip

Let us now show that $\reg S/I(X)\geq (q-2)(k-1)$.
If we show that $h_d\neq0$, i.e., $h_d>0$, for $d= (q-2)(k-1)$,
then $H_X(d-1)<H_X(d)$, for $d= (q-2)(k-1)$,
and our result follows.
%For which, $R/(t_1)$ being $0$-dimensional, it 
It suffices to produce a monomial $M$ of degree $d=(q-2)(k-1)$,
such that $M \in (S')_d$ but $M \notin (I'(X))_d$.
Consider
\[
M = (t_2\cdots t_{k})^{q-2} \; 	\in \;(S')_d.
\]
Suppose $M \in I'(X)$. Then, as we have seen above,
$$
M = \sum_{j=2}^s g_j t_j^{q-1} + \sum_{\sigma,r} h_{\sigma,r} t^b
$$
where $g_j, h_{\sigma,r} \in S'$, $f^r_\sigma=t^a-t^b$ and the second summation runs over all 
partitions $\sigma$ of $\set{1,\dots,s}$ into $2$ parts of equal
cardinality and $r\in \set{1,\dots,q-2}$. Since $M$ is a monomial and its degree in each one 
of the variables is $q-2$, we deduce that $M$ must be a monomial of the form
$$
M= h_{\sigma,r}t^b
$$
for $h_{\sigma,r} \in S'$, one partition $\sigma$ of $\set{1,\dots,s}$ into $2$ parts of equal
cardinality, one $r\in \set{1,\dots,q-2}$ and $f^r_\sigma=t^a-t^b$. But this is not possible because 
the monomial $M$ has $k-1$ variables, while $h_{\sigma,r}t^b$ has at least $k$ variables.
We conclude that $M \notin I'(X)$.
%Then, $M=0$ in $R/(t_1)$ if and only if there exists $A\in S$ such that 
%$M+At_1\in I(X)$. By Theorem~\ref{theorem: the generators of I(X)}, there 
%exist $B_{i,j},B_{\sigma,r}\in S$ such that 
%\[
%M + At_1 = \sum_{i<j} B_{i,j}(t_i^{q-1}-t_j^{q-1}) + \sum_{\sigma,r} B_{\sigma,r} f^r_\sigma 
%\]
%where the second summation runs over all partitions $\sigma$ of $\set{1,\dots,s}$ into $2$ parts of equal
%cardinality and $r\in \set{1,\dots,q-2}$. Now, for any monomial $c_at^a$, with $c_a\in K$, 
%resulting from the summations on the right hand side, either there exists $j\in \set{1,\dots,s}$ 
%such that $t_j^{q-1}$ divides $c_at^a$ or $\ls{\supp(a)}\geq k$. Since $M$ does not cancel in $M+At_1$, as 
%$t_1$ does not divide it, we deduce that $M$ must be one such monomial. However it satisfies none of the 
%previous conditions. We conclude that $M\not = 0$ in $R/(t_1)$; which means that 
%$\reg R/(t_1)\geq (q-2)(k-1)+1$.
\end{proof}

\begin{theorem}\label{theorem: bound on regularity for a general graph}
Let $\G$ be a bipartite graph. Let $\H_1,\dots,\H_m$ be subgraphs of
$\G$ isomorphic to $($even$)$ cycles $\H_i\cong \C_{2k_i}$ that have
disjoint edge sets. Then 
\[
\textstyle \reg S/I(X) \leq (q-2)\bigr(s-\sum_{i=1}^m k_i-1\bigl).
\]
\end{theorem}

\begin{proof}
For all $1\leq i\leq m$, let $t_{\epsilon^i_1},\dots,t_{\epsilon^i_{2k_i}}$ 
be the variables associated to the edges of $\H_i$.
%We assume, without loss of generality that $t_i$ is one of the variables associated to the edges of 
%$\H_i$, for all $1\leq i\leq m$.
 
Without loss of generality, assume that $t_1=t_{\epsilon^1_1}, \;t_2=t_{\epsilon^2_1},\; \ldots,
\;t_m=t_{\epsilon^m_1}$.

Denote by $R$ the quotient $S/I(X)$ and, for $1\leq i\leq m$, let 
\[
R_i = R /(t_1,\dots,t_i).
\]
Since $t_1$ is a regular element of $R$, we have the following short exact sequence 
of graded $S$-modules:
\begin{equation}\label{eq: SES}
 0\longrightarrow R[-1] \stackrel{t_1}{\longrightarrow} R \longrightarrow R_1 \longrightarrow 0 .
\end{equation}
Furthermore, for all $1\leq i \leq m-1$, we have exact sequences of graded $S$-modules:
\begin{equation}\label{eq: amputated SESs}
R_i[-1] \stackrel{t_{i+1}}{\longrightarrow} R_i \longrightarrow R_{i+1} \longrightarrow 0 .
\end{equation}

\begin{claim1}
For all $1\leq i\leq m$, $t_j^{q-1}=0$ in $R_i$, for all $1\leq j\leq s$.
\end{claim1}
\begin{proof}[Proof of Claim 1]
Since $t_j^{q-1}-t_i^{q-1}\in I(X)$ and $t_i^{q-1}=0$ in $R_i$, we deduce that $t_j^{q-1}=0$ in $R_i$, for 
all $1\leq j\leq s$.
\end{proof}
\begin{claim2}
If there exists a nonnegative integer $\ell$ such that $(R_{i+1})_d=0$, for all $d\geq \ell$, then
$(R_i)_d=0$ for all $d\geq \ell+q-2$, \;where $1\leq i\leq m-1$.
\end{claim2}
\begin{proof}[Proof of Claim 2]
If $(R_{i+1})_d=0$, for $d\geq \ell$ then from (\ref{eq: amputated SESs}) we deduce that for all
$d\geq \ell$ the maps $(R_i)_{d-1} \stackrel{t_{i+1}}{\longrightarrow} (R_i)_d$
are surjective, \emph{i.e.}, $(R_i)_d = t_{i+1}(R_i)_{d-1}$, for all $d\geq
\ell$. Iterating and using Claim~1, we get: 
$(R_i)_{d+q-2} = t_{i+1}^{q-1}(R_i)_{d-1}=0$,  
\emph{i.e.}, $(R_i)_d=0$ for all $d\geq \ell+q-2$.
\end{proof}
\begin{claim3}
Let $t^a$ be a monomial in $S$. Suppose that the degree of $t^a$ in the variables associated to $\H_i$ is 
$\geq (q-2)(k_i-1)+1$. Then $t^a=0$ in $R_i$.
\end{claim3}
\begin{proof}[Proof of Claim 3]
We may assume that $t_i$ does not divide $t^a$. 
%Let $t_{\epsilon^i_1},\dots,t_{\epsilon^i_{2k_i}}$ 
%be the variables associated with 
%the cycle $\H_i$, with $t_{\epsilon^i_1}=t_i$. 
Defining
\[S_i:= K\big[t_{\epsilon^i_1},\dots,t_{\epsilon^i_{2k_i}}\big] ,
\] 
we have
$I(X_i) \subset S_i$, where $X_i \subseteq \PP^{2k_i-1}$ is the set of points
parameterized by the edges of the cycle $\H_i$. It is
straightforward to check that $I(X_i)\subset I(X) \subset S$. Let $t^a=t^bt^c$, where $t^b$ is a monomial
in $t_{\epsilon^i_1},\dots,t_{\epsilon^i_{2k_i}}$. It suffices to
show $t^b=0$ in $S_i/(I(X_i)+t_i)$, but since $t^b$ has  
degree $\geq (q-2)(k_i-1)+1$, we can run the same argument as in the proof of 
Theorem~\ref{theorem: regularity of even cycles}.
\end{proof}

\begin{claim4}
Let $\ell_0=(q-2)\bigr(\sum_{i=1}^m(k_i-1)\bigl) + (q-2)\bigr(s-\sum_{i=1}^m 2k_i\bigl) + 1$. Then 
$(R_m)_d =0$, $\forall\; d\geq \ell_0$. 
\end{claim4}

\begin{proof}[Proof of Claim 4]
Let $t^a$ be a monomial of degree $d\geq \ell_0$. In view of Claim~3, we may assume that the degree 
of $t^a$ in the variables associated to $\H_i$ is $\leq
(q-2)(k_i-1)$. Then, the degree of $t^a$ in the remaining 
$s-\sum_{i=1}^m 2k_i$ variables is $\geq (q-2)\bigr(s-\sum_{i=0}^m 2k_i\bigl) + 1$ which implies that
one of them is raised to a power $\geq q-1$ and therefore, by Claim~1, $t^a=0$ in $R_m$.
\end{proof}

\noindent
We now finish the proof of the theorem. Notice that 
$\ell_0=(q-2)\bigr( s -\sum_{i=1}^m (k_i+1)\bigr)+1$. Combining Claim~2 with Claim~4 we
deduce that $(R_1)_d = 0$, for all $d\geq \ell_0+(m-1)(q-2)$. 
Now $\ell_0+(m-1)(q-2) = (q-2)\bigr(s-\sum_{i=1}^m k_i-1\bigl)+1$ and using (\ref{eq: SES})
we see that $(R)_{d-1}\stackrel{t_1}{\longrightarrow} (R)_d$ is an isomorphism for all 
$d\geq (q-2)\bigr(s-\sum_{i=1}^m k_i-1\bigl)+1$. This means that the Hilbert function of $R$
satisfies:
$H_X(d-1)=H_X(d)$,
for $d-1 \geq (q-2)\bigr(s-\sum_{i=1}^m k_i-1\bigl)$. 
Hence, $\reg R\leq (q-2)\bigr(s-\sum_{i=1}^m k_i-1\bigl)$.
\end{proof}

\begin{remark}
Notice we do not assume that $\G$ is connected nor do we assume that 
any $2$ cycles, $\H_1$ and $\H_2$, in $\G$ have  
disjoint edge or vertex sets. In fact, we can apply the bound of 
Theorem~\ref{theorem: bound on regularity for a general graph}
to both graphs in Figure~\ref{fig: two offending graphs}. For $\G_1$, on the left, we should 
use both cycles of order $4$. 
We obtain \mbox{$\reg S/I(X_1)\leq 3(8-4-1)=9$}. Using {\it
Macaulay}$2$ \cite{mac2}, for $q=5$, 
we checked that this is the actual value of the regularity.
For $\G_2$, on the right, we may only use 
one of the cycles. Then, Theorem~\ref{theorem: bound on regularity for a general graph} yields
$\reg S/I(X_2) \leq (q-2)(6-2-1)=3(q-2)$, which, for $q=5$, is not sharp, as the value of 
$\reg S/I(X_2)$ is $6$. The inequality of Theorem~\ref{theorem: bound on regularity for a general graph}
is an improvement of the inequality given in \cite[Corollary~2.31]{deg-and-reg}.
\end{remark}

\begin{corollary}\label{corollary: regularity for almost general graphs}
Let $\G$ be a connected bipartite graph, the $($even$)$ cycles of which, 
$\H_1,\dots,\H_m$, with $\H_i\cong \C_{2k_i}$, 
have disjoint vertex sets. Then 
\[
\textstyle \reg S/I(X)= (q-2)\bigr(s-\sum_{i=1}^m k_i-1\bigl).
\]
\end{corollary}

\begin{proof} Let $t_{\epsilon^i_1},\dots,t_{\epsilon^i_{2k_i}}\hspace{-.2cm}\in S$
be the set of variables 
associated to the edges, $e^i_1,\dots,e^i_{2k_i}$ of the even cycle
$\H_i$. We set 
\[
S_i=K\bigl[t_{\epsilon^i_1},\dots,t_{\epsilon^i_{2k_i}}\bigr]\subset
S,
\]
and denote by $I_i(X)$ the intersection $I(X)\cap S_i$. Then, 
$I_i(X)\subset S_i$ is the vanishing ideal of the algebraic toric set
associated to $\H_i$. By Theorem~\ref{theorem: I(X) for almost general graph}, 
$I(X)$ is generated by the set
\[
\mathcal{J}=\{t_i^{q-1}-t_j^{q-1}:1\leq i,j\leq s\} \cup  I_1(X)\cup\cdots\cup I_m(X).
\]
We proceed by induction on the number of edges of $\G$. 
If $\G$ is an even cycle, the result follows from
Theorem~\ref{theorem: regularity of even cycles}. We may assume that
$e_s$ is an edge of $\G$ that does not lie on any cycle of $\G$ and
that $t_s$ is the variable that corresponds to $e_s$. For simplicity
of notation, we identify the edge $e_i$ with the variable $t_i$ for
$i=1,\ldots,s$ and refer to $t_i$ as an edge of the graph $\G$.  
Consider the
graph $\G_1$ whose edge set is $\{e_1, \ldots, e_{s-1}\}$ (the edge set of $\G$ minus the edge
$e_s$), and whose vertex set is the set of endpoints of the edges $e_1, \ldots, e_{s-1}$.
%union of the edges of $\G$ different from $e_s$. 
Let $X_1$ be the algebraic toric set parameterized by the
edges of $\G_1$. Clearly $\G_1$ is a bipartite graph whose (even)
cycles are again $\H_1,\dots,\H_m$.

\medskip

\noindent
Case (I): The graph $\G_1$ is connected. Let
$A(X_1)=K[t_1,\ldots,t_{s-1}]/I(X_1)$ be the coordinate
ring of $X_1$ and let $F_{X_1}(t)$ be the Hilbert series of $A(X_1)$.
The Hilbert series can be uniquely written as
$F_{X_1}(t)=g_1(t)/(1-t)$, where $g_1(t)$ is a polynomial of degree
equal to the regularity of $A(X_1)$. Because $\G_1$ is a connected bipartite graph 
and has the same cycles as $\G$, by Theorem~\ref{theorem: I(X)
for almost general graph}, the vanishing ideal $I(X_1)$ is generated
by the set
\[
\mathcal{J}_1=\{t_i^{q-1}-t_j^{q-1}:1\leq i,j\leq s-1\} \cup
I_1(X)\cup\cdots\cup I_m(X) 
\]
(notice that $I_j(X)=I_j(X_1)$, for $j= 1, \ldots, m$).
Hence, there is an exact sequence 
\begin{small}
\[
0\rightarrow A(X_1)[-(q-1)]\stackrel{\scriptstyle \hspace{.1cm}t_1^{q-1}}{\longrightarrow}
A(X_1)\longrightarrow
C=K[t_1,\ldots,t_{s-1}]/(I_1(X),\ldots,I_m(X),t_1^{q-1},\ldots,t_{s-1}^{q-1})
\rightarrow 0.
\]
\end{small}
\noindent As a consequence, we get that the Hilbert series $F(C,t)$ of $C$ is given by 
$$
F(C,t)=F_{X_1}(t)(1-t^{q-1})=g_1(t)(1+t+\cdots+t^{q-2}),
$$
%Therefore, as $\G_1$ is connected and bipartite, by induction we get 
and $\textstyle \deg\, F(C,t)=(q-2)+{\rm reg}\, A(X_1)$. 
Since $\G_1$ is a connected bipartite graph ant its even cycles have disjoint vertex sets,
by induction we get ${\rm reg}\, A(X_1)=(q-2)\bigr(s-1-\sum_{i=1}^m k_i-1\bigl)$, and therefore,
\begin{equation}\label{nov2-11}
\textstyle \deg\, F(C,t)=(q-2)\bigr(s-\sum_{i=1}^m k_i-1\bigl). 
\end{equation}
From the exact sequence 
$$
0\rightarrow (S/I(X))[-1]\stackrel{t_s\ }{\longrightarrow}
S/I(X)\longrightarrow
S/(t_s,I(X))\rightarrow 0,
$$
we get that $F_X(t)=F(S/(t_s,I(X)),t)/(1-t)$. Thus ${\rm reg}(S/I(X))=
\deg\, F(S/(t_s,I(X)),t)$. Using the isomorphism
$$
S/(t_s,I(X))\simeq K[t_1,\ldots,t_{s-1}]/(t_1^{q-1},\ldots,t_{s-1}^{q-1},
I_1(X),\ldots,I_m(X)),
$$
we obtain that $C\simeq S/(t_s,I(X))$. Hence, by Eq.~(\ref{nov2-11}), the desired formula
follows. 

\medskip

\noindent
Case (II): The graph $\G_1$ is disconnected. It is not hard to show
that $\G_1$ has exactly two connected components $\G_1'$, $\G_1''$.
Let $E_1'$, $E_1''$ be the edge sets of $\G_1'$, $\G_1''$
respectively and let $X_1'$, $X_1''$ be the algebraic toric sets
parameterized by the edges of $\G_1'$, $\G_1''$ respectively. We may
assume that $\H_1,\ldots,\H_{r}$ are the cycles of $\G_1'$ and
$\H_{r+1},\ldots,\H_{m}$ are the cycles of $\G_1''$. By Theorem~\ref{theorem: I(X)
for almost general graph}, we have that $I(X_1')$ and $I(X_1'')$ are
generated by  
\begin{eqnarray*}
\mathcal{J}_1'=\{t_i^{q-1}-t_j^{q-1}:\, t_i,t_j\in E_1'\} \cup
I_1(X)\cup\cdots\cup I_r(X)\ \mbox{ and }\ \ \ \ \ \ \ \ \ \ \ \ \ \
\ \ \ \ & &\\ 
 \mathcal{J}_1''=\{t_i^{q-1}-t_j^{q-1}:\, t_i,t_j\in E_1''\} \cup
I_{r+1}(X)\cup\cdots\cup I_m(X), & &
\end{eqnarray*}
respectively. We set 
$$
C_1'=K[E_1']/(\{t_i^{q-1}\}_{t_i\in
E_1'},I_1(X),\ldots,I_r(X)),\ \ C_1''=K[E_1'']/(\{t_i^{q-1}\}_{t_i\in
E_1''},I_{r+1}(X),\ldots,I_m(X)).
$$
By the arguments that we used to prove Case (I), and using the induction hypothesis, we
get 
$$
\textstyle \deg F(C_1',t)=(q-2)\bigr(|E_1'|-\sum_{i=1}^r
k_i\bigl),\quad \deg F(C_1'',t)=(q-2)\bigr(|E_1''|-\sum_{i=r+1}^m k_i\bigl).
$$
Since $K[E_1']$ and $K[E_1'']$ are polynomial rings in disjoint sets
of variables $E_1'$ and $E_1''$, according to
\cite[Proposition~2.2.20, p.~42]{monalg}, we have an isomorphism 
$$
C_1'\otimes_K
C_1''\simeq
K[t_1,\ldots,t_{s-1}]/(t_1^{q-1},\ldots,t_{s-1}^{q-1},I_1(X),\ldots,I_m(X))=S/(t_s,I(X)).
$$
Altogether, as $F(C_1'\otimes_K C_1'',t)=F(C_1',t)F(C_1'',t)$ (see
\cite[p.~102]{monalg}), we obtain 
\begin{eqnarray*} 
\reg S/I(X)&=&\deg\, F(S/(t_s,I(X)),t)=\deg F(C_1'\otimes_K
C_1'',t)=\deg F(C_1',t)+\deg F(C_1'',t)\\
&=&\textstyle(q-2)\bigr(|E_1'|+|E_1''|-\sum_{i=1}^m
k_i\bigl)=(q-2)\bigr(s-\sum_{i=1}^m
k_i-1\bigl),
\end{eqnarray*}
as required. This completes the proof of case (II).
\end{proof}

\bibliographystyle{plain}

\end{document}